\documentclass[11pt]{amsart}
\usepackage{fullpage, amssymb, eepic}

\newtheorem{theorem}{Theorem}
\newtheorem{proposition}[theorem]{Proposition}
\newtheorem{lemma}[theorem]{Lemma}
\newtheorem{corollary}[theorem]{Corollary}

\theoremstyle{definition}

\newtheorem{question}[theorem]{Question}

\newtheorem{example}[theorem]{Example}

\theoremstyle{remark}
\newtheorem{remark}[theorem]{Remark}

\numberwithin{equation}{section}
\numberwithin{theorem}{section}

\newcommand{\A}{\mathcal{A}}

\newcommand{\B}{\mathcal{B}}
\newcommand{\C}{\mathfrak{c}}
\newcommand{\D}{\mbox{\textnormal{dens}}}
\newcommand{\Dc}{\mbox{\textnormal{dec}}}
\newcommand{\F} {\mathbb{F}}

\newcommand{\G}{\mbox{\textnormal{gen}}}
\newcommand{\GG}{\mathcal{G}}
\newcommand{\h}{\mathfrak{H}}

\newcommand{\J}{\mathcal{J}}
\newcommand{\kH}{\mathfrak{K}}

\newcommand{\M}{\mathcal{M}}
\newcommand{\N}{\mathcal{N}}

\newcommand{\Q}{\mathbb{Q}}
\newcommand{\R}{\mathcal{R}}

\newcommand{\X}{\mathfrak{X}}
\newcommand{\z}{\mathcal{Z}}

\begin{document}

\title[Generators of von Neumann algebras]{On cardinal invariants and generators \\ for von Neumann algebras }
\author{David Sherman}
\address{Department of Mathematics\\ University of Virginia\\ P.O. Box 400137\\ Charlottesville, VA 22904, USA}
\email{dsherman@virginia.edu}

\date{\today}

\begin{abstract}
We demonstrate how most common cardinal invariants associated to a von Neumann algebra $\M$ can be computed from the decomposability number, $\Dc(\M)$, and the minimal cardinality of a generating set, $\G(\M)$.  %This includes various density characters and the cardinality of $\M_*$; it does not include the cardinality of $\M$, although there is a simple formula that works unless $\M$ is extremely large.
Applications include the equivalence of the well-known generator problem, ``Is every separably-acting von Neumann algebra singly-generated?", with the formally stronger questions, ``Is every countably-generated von Neumann algebra singly-generated?" and ``Is the $\G$ invariant monotone?"  Modulo the generator problem, we determine the range of the invariant $(\G(\M), \Dc(\M))$, which is mostly governed by the inequality $\Dc(\M) \leq \C^{\text{gen}(\M)}$.
\end{abstract}

\maketitle

\section{Introduction}
In this paper we consider various ways of describing the size of a von Neumann algebra $\M$.  We show that most common cardinal invariants can be computed in terms of the \textit{minimal cardinality of a generating set}, $\G(\M)$, and the \textit{decomposability number}, $\Dc(\M)$.  For example, their product is the \textit{representation density}, $\chi_r(\M)$ (Theorem \ref{T:estimate}(2)).  (See the next section for definitions.)  With $\C$ the cardinality of the continuum, always $\Dc(\M) \leq \C^{\text{gen}(\M)}$ (Theorem \ref{T:estimate}(2)); this essentially determines the range of the invariant $(\G(\M), \Dc(\M))$ (Theorem \ref{T:range}). We give a formula for computing $\G$ of an arbitrary direct sum (Theorem \ref{T:dsum}) and deduce that the condition $\Dc(\M) > \aleph_0 \cdot \G(\M)$ can only hold when the center is large (Proposition \ref{T:center}(1)).  We also show that $\Dc(\M)$ and $\G(\M)$ determine the cardinality of $\M_*$, but not of $\M$, although the formula $|\M|= (\aleph_0 \cdot \G(\M))^{\aleph_0 \cdot \text{dec}(\M)}$ works as long as $\M$ can be written as a direct sum of algebras each of which can be generated by fewer than $(2^{\aleph_1})^{+\omega_1}$ elements (and this cardinal bound is sharp).

One of our underlying motivations is to give new formulations of the generator problem for von Neumann algebras, which we briefly describe now.

There are many criteria by which a von Neumann algebra may be considered ``small."  One is separability of the predual; this is equivalent to the existence of a faithful representation on $\ell^2$.  We will call such algebras ``separably-acting."  Another criterion for smallness is the presence of a countable generating set, or even better, the presence of a single generator.

\begin{question} \label{T:gen} (The generator problem)
Is every separably-acting von Neumann algebra singly-generated?
\end{question}

Every separably-acting von Neumann algebra is countably-generated, but the converse is not true.  For example, the atomic abelian von Neumann algebra $\ell^\infty_\C$ is generated by any single element whose components are all distinct, and its predual $\ell^1_\C$ is nonseparable.  Thus the following question is formally stronger.

\begin{question} \label{T:gen2}
Is every countably-generated von Neumann algebra singly-generated?
\end{question}

We will see that the two questions are actually equivalent (Theorem \ref{T:main}), so that either may be termed ``the generator problem."  We also show that Questions \ref{T:gen} and \ref{T:gen2} are equivalent to asking whether $\G$ is monotone (Theorem \ref{T:monotone}(3)) or multiplicative on tensor products (Corollary \ref{T:mult}(3)).  Unfortunately we offer little insight here into the answers to these questions, other than the fact that they are identical.  Over the years more and more classes of separably-acting von Neumann algebras have been shown to be singly-generated, including those that are type I (\cite{P}) or properly infinite (\cite[Theorem 2]{Wo}).  It is also known that a full positive answer would follow from a positive answer for $\text{II}_1$ factors (\cite[Corollary 2]{W}) -- here we add the possibly useful observation that one can restrict attention to finitely-generated $\text{II}_1$ factors (Theorem \ref{T:finite}).  On the other hand there has been feeling that free entropy and other tools from free probability might show that algebras such as $L(\F_3)$ are counterexamples.  For more on the current status of the generator problem for $\text{II}_1$ factors, the reader could consult \cite[Chapter 16]{SS2} or \cite{GS}.

The paper is structured as follows.  In the next section we establish a number of relations between invariants that measure the size of a von Neumann algebra.  In Section \ref{S:main} we prove that Questions \ref{T:gen} and \ref{T:gen2} are equivalent and use Shen's invariant $\GG(\cdot)$ to further reduce to the finitely-generated case.  Section \ref{S:sums} establishes the formula $\G(\sum^\oplus \M_i) = \max \{\log_\C (|I|), \sup \G(\M_i)\}$, then identifies (modulo the generator problem) the pairs of cardinals that arise as $(\G(\M), \Dc(\M))$.  In Section \ref{S:center} we consider what cardinal invariants can say about the center, or about the algebra modulo the center, and we generalize some results of Kehlet.  Section \ref{S:mm} proves that the generator problem is equivalent to monotonicity of $\G$, or multiplicativity of $\G$ on tensor products.  Section \ref{S:dd} comments on the invariants of double duals of $C^*$-algebras, and responds (not quite completely) to some questions of Hu and Neufang.  In the final section we investigate when and how $\G(\M)$ and $\Dc(\M)$ determine the cardinality of $\M$.

Owing to the quantity of invariants, it can be difficult even for experts to keep the interdependences straight.  A secondary goal of this paper is simply to collect and organize all the relevant information, including examples and some brief historical discussion.

None of the results in this paper rely on set theoretic assumptions beyond ZFC.

\section{Describing the size of a von Neumann algebra} \label{S:size}

Representations of von Neumann algebras are always understood here to be normal.  The symbol ``$\simeq$" stands either for *-isomorphism of von Neumann algebras or isometric isomorphism of Banach spaces.  The center of a von Neumann algebra $\M$ is $\z(\M)$, and in any direct sum $\sum^\oplus \M_i$ we let $\{e_i\}$ be the coordinate projections.

The cardinality of a set $S$ is $|S|$.  The \textit{density character} of a topological space is the minimal cardinality of a dense set, and the norm density character of a Banach space $\X$ will be denoted $\D(\X)$.  For a Hilbert space $\h$, we have $\D(\h) = \aleph_0 \cdot \text{dim}(\h)$: consider finite linear combinations of basis elements over $\Q + i\Q$.  We also write $s$-$\D$ for the density character of a von Neumann algebra $\M$ or its unit ball $\M_{\leq 1}$ with respect to the $\sigma$-strong topology.  The reader should be aware that in general (nonmetrizable) Hausdorff spaces the density character may increase when passing to a subspace, even a closed subgroup of a topological group (see \cite{CI} for examples and discussion).  It will turn out that this phenomenon does not occur in the situations considered in this paper.

Here are three cardinal invariants for a von Neumann algebra $\M$.
\begin{itemize}
\item $\G(\M)=$ minimal cardinality of a generating set.  %At a few places in the text we will use the variation $\HG(\M)=$ minimal cardinality of a set of Hermitian generators (equal to either $2 \cdot \G(\M)$ or $2\cdot \G(\M)-1$).
    By fiat we set $\G(\mathbb{C})=1$ instead of 0.
\item $\chi_r(\M)=$ minimal dimension of a Hilbert space on which $\M$ acts faithfully.  We take this notation and the name \textit{representation density} from \cite[Section 7]{FK}, where the $C^*$-version is briefly developed.  In Theorem \ref{T:estimate}(2) we show that $\chi_r(\M) = \D(\M_*)$ whenever $\M$ is infinite-dimensional, which generalizes the often-mentioned, rarely-proved fact that a von Neumann algebra is separably-acting if and only if it has separable predual (e.g., \cite[Lemma 1.8]{Y}).
\item $\Dc(\M)=$ maximal cardinality of a set of pairwise orthogonal nonzero projections in $\M$.  (That the supremum is achieved is proved in \cite[Theorem 2.6(i)]{HN}.)  This notation, for \textit{decomposability number}, is taken from the series of papers \cite{HN,Hu,N}, although the concept had appeared earlier in \cite[p.54]{AA}.  Of course it is motivated by the condition called either $\sigma$-finiteness or countable decomposability, which amounts to $\Dc(\M) \leq \aleph_0$.  %For this quantity at least one standard fact has already been generalized: when $\M$ is represented faithfully on a Hilbert space, $\aleph_0 \cdot \Dc(\M)$ equals $\aleph_0$ times the minimal cardinality of a separating set of vectors (\cite[Proposition 2.8]{N}).
\end{itemize}

It is classical that a von Neumann algebra $\M$ acts faithfully on a separable Hilbert space if and only if it is both countably-generated and $\sigma$-finite (\cite[Exercice I.7.3bc]{D}).  In other words,
\begin{equation} \label{E:int}
\chi_r(\M) \leq \aleph_0 \iff [\G(\M) \leq \aleph_0 \text{ and } \Dc(\M) \leq \aleph_0].
\end{equation}
In Theorem \ref{T:estimate}(2) we will obtain the general statement $\chi_r(\M) = \G(\M) \cdot \Dc(\M)$.

One reason \eqref{E:int} is easy to misremember is that the analogous conditions for $C^*$-algebras interact in a totally different manner: countable generation is equivalent to separability (of the algebra), and this is strictly stronger than being representable on a separable Hilbert space.  Figure 1 is intended to help the reader visualize \eqref{E:int} and its relation to our treatment of the generator problem.  Most von Neumann algebras one encounters are in $C$, and we have already mentioned that the algebra $\ell^\infty_\C$ belongs to $B$.  We will describe several inhabitants of $E$ in Example \ref{T:LG}.  The usual generator problem (Question \ref{T:gen}) asks whether $D$ is empty, while Question \ref{T:gen2} asks whether $A$ and $D$ are both empty.

\setlength{\unitlength}{.5in}
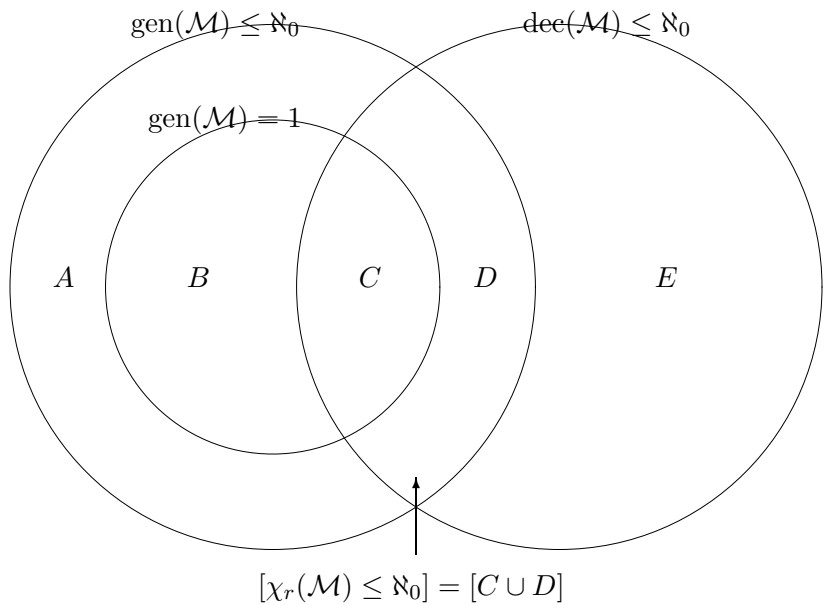
\begin{figure}[b] \label{F:venn}

\begin{picture}(12,6.1)
\put(4.5,3.3){\circle{5.5}}
\put(4.5,3.3){\circle{3.5}}
\put(7.5,3.3){\circle{5.5}}

\put(8,5.9){\makebox(0,0)[b]{$\Dc(\M) \leq \aleph_0$}}
\put(3.9,5.9){\makebox(0,0)[b]{$\G(\M) \leq \aleph_0$}}
\put(4,4.9){\makebox(0,0)[b]{$\G(\M) =1$}}
\put(6,.5){\vector(0,1){.8}}
\put(6,0){\makebox(0,0)[b]{$[\chi_r(\M) \leq \aleph_0] = [C \cup D]$ }}

\put(2.2,3.3){$A$}
\put(3.6,3.3){$B$}
\put(5.4,3.3){$C$}
\put(6.6,3.3){$D$}
\put(8.5,3.3){$E$}

\end{picture}
\caption{The ``small" von Neumann algebras described in \eqref{E:int}.  Question \ref{T:gen} asks whether $D$ is empty.  Question \ref{T:gen2} asks whether $A$ and $D$ are both empty.}
\end{figure}

The next theorem shows how several cardinal invariants for von Neumann algebras are related.  Some special cases were noted in work of Hu and Neufang (e.g., \cite[Proposition 3.2]{Hu} and \cite[Corollary 2.7]{HN}); their emphases were different and are briefly discussed in Section \ref{S:HN}.

\begin{theorem} \label{T:estimate}
Let $\M$ be a von Neumann algebra.
\begin{enumerate}
\item One can write $\M$ as a direct sum $\sum^\oplus_{i \in I} \M_i$, where $|I| \leq \Dc(\z(\M)) \leq \Dc(\M)$ and for each $i$,
$$\chi_r(\M_i) \leq \aleph_0 \cdot \G(\M) = s\text{-}\D(\M) = s\text{-}\D(\M_{\leq 1}).$$
\item The following relations hold:
\begin{equation*} \label{E:estimate}
\tag{$\heartsuit$} \G(\M) \cdot \Dc(\M) = \chi_r(\M) \leq \aleph_0 \cdot \chi_r(\M) = \D(\M_*) \leq |\M_*| = \C \cdot \chi_r(\M)^{\aleph_0} \leq \C^{\textnormal{gen}(\M)}.
\end{equation*}
\end{enumerate}
Thus $\G(\M)$ and $\Dc(\M)$ together determine $s$-$\D(\M)$, $\chi_r(\M)$, $\D(\M_*)$, and $|\M_*|$.
\end{theorem}

\begin{proof} We first dispose of the case where $\M$ is finite-dimensional.  Then $\M$ is of the form $\sum^\oplus_k \mathbb{M}_{n_k}$, with $\G(\M) = 1$ and $\chi_r(\M) = \Dc(\M) = \sum_k n_k$.  All claims of the theorem are easily verified.  \textit{For the remainder of the proof we assume that $\M$ is infinite-dimensional}, so that $\chi_r(\M)$ and $\Dc(\M)$ are necessarily infinite (\cite[Proposition 2.5]{HN}).

(1) Let $\M$ be generated by $\{x_\alpha\}_{\alpha < \text{gen}(\M) }$.  Set $\A_0$ to be the $\sigma$-strongly dense subset of $\M$ consisting of noncommuting *-polynomials in the $x_\alpha$ with coefficients in $\Q + i\Q$.  Because any $\sigma$-strongly dense set is infinite and generating, we have
$$s\text{-} \D(\M) \leq |\A_0| \leq \aleph_0 \cdot \G(\M) \leq \aleph_0 \cdot s\text{-} \D(\M) = s\text{-} \D(\M).$$

As mentioned earlier, it is in general false that the density character of a topological space dominates the density character of a subspace, so we need a short argument to establish that $s$-$\D(\M_{\leq 1})$ also equals $|\A_0|=s\text{-}\D(\M)$.  The Kaplansky density theorem implies that $\M_{\leq 1} \cap \A_0$ is $\sigma$-strongly dense in $\M_{\leq 1}$, giving $s$-$\D(\M_{\leq 1}) \leq |\M_{\leq 1} \cap \A_0| = |\A_0|$.  On the other hand, if $S$ is any $\sigma$-strongly dense set in $\M_{\leq 1}$, then the set of positive rational multiples of elements of $S$ (which has the same cardinality as $S$) is $\sigma$-strongly dense in $\M$: this gives $s$-$\D(\M_{\leq 1}) \geq s\text{-}\D(\M)$.

Now represent $\M$ on a Hilbert space $\h$ and choose any $0 \neq \xi \in \h$.  The space $\h_0 = \overline{\M \xi} = \overline{\A_0 \xi}$ is $\M$-invariant and clearly has density character $\leq |\A_0| = \aleph_0 \cdot \G(\M)$.  Since $\M$ is represented normally (but not necessarily faithfully) on $\h_0$, the image of $\M$ is isomorphic to $z\M$ for some central projection $z \in \M$.  %Now $\xi$ is cyclic for $z\M$, so separating for its commutant, and in particular separating for $\z(z\M)$.  This entails that $z$ is $\sigma$-finite in $\z(\M)$ (\cite[Proposition II.3.19]{T}).

By Zorn's lemma $\h$ can then be decomposed as a sum of $\M$-invariant subspaces $\{\h_i\}_{i \in I}$ with $\dim \h_i \leq \aleph_0 \cdot \G(\M)$.  Write $\M|_{\h_i} \simeq z_i \M$.  Totally order the index set, and define $y_i = z_i - \vee_{j < i} z_j$.  Set $I' = \{i \in I \mid y_i \neq 0\}$ and $\M_i = y_i \M$ for $i \in I'$, so $\{y_i\}_{i \in I'}$ are nonzero central projections summing to 1 and $\M \simeq \sum_{I'}^\oplus y_i \M$.  By definition $|I'| \leq \Dc(\z(\M))$.  Also $\chi_r(\M_i) \leq \aleph_0 \cdot \G(\M)$, since $\M_i$ can be represented on a subspace of $\h_i$.

%The only detail left to address is that $|I'|$ may be strictly less than $\Dc(\z(\M))$.  By \cite[Theorem 2.6]{HN} this can only happen if $|I'|$ is finite and $\Dc(\z(\M))$ is countable, in which case there is a new family of $\Dc(\z(\M))$ nonzero orthogonal projections $\{z''_n\}$ such that each $z'_i$ is the sum of projections from $\{z''_n\}$.  As subprojections of the $z'_i$, the $z''_n$ are $\sigma$-finite and still satisfy $\chi_r(z''_n \M) \leq \aleph_0 \cdot \G(\M)$.

(2) We treat each nontrivial relation separately.

\underline{$\G(\M) \leq \chi_r(\M)$}: Since $\G(\M) \leq \aleph_0 \cdot \G(\M) = s\text{-}\D(\M_{\leq 1})$ from part (1), it suffices to prove that $s\text{-} \D(\M_{\leq 1}) \leq \kappa$ whenever $\M \subseteq \B(\ell^2_\kappa)$.  We effectively show that $s\text{-} \D(\M_{\leq 1}) \leq s\text{-} \D(\B(\ell^2_\kappa)_{\leq 1})$.  Later (Theorem \ref{T:monotone}(2)) we will combine this fact with others to obtain the same conclusion for any inclusion of von Neumann algebras.  %Again one cannot appeal to a principle that the density character of a subspace is no larger than that of the space, at least not without some consideration of the \textit{weight} of the space (=minimal cardinality of a basis for the topology).  We handle the case before us explicitly.

Fix a basis $\{\xi_\beta\}_{\beta < \kappa}$ for $\ell^2_\kappa$.  Let $\{x_\alpha\}_{\alpha < \kappa} \subset \B(\ell^2_\kappa)_{\leq 1}$ be a $\sigma$-strongly dense set: for example, one can take the contractive operators whose matrices have finitely many nonzero entries taking values in $\Q + i\Q$.  The $\sigma$-strong topology on $\B(\ell^2_\kappa)_{\leq 1}$ is just the strong topology, generated by the seminorms $p_\beta(y) = \|y \xi_\beta\|$.  Consider the $\kappa$ strongly open subsets of
$\B(\ell^2_\kappa)_{\leq 1}$
$$V_{\alpha, F, n} = \{y\mid p_\beta(y - x_\alpha) < 1/n \text{ for all $\beta$ in the finite set of indices $F$}\}.$$
For each multi-index $(\alpha, F, n)$, choose an element $y_{\alpha, F, n} \in \M_{\leq 1} \cap V_{\alpha, F, n}$ if the intersection is nonempty.  We claim that the set of $\leq \kappa$ elements chosen is strongly dense in $\M_{\leq 1}$.

For the claim, it suffices to take any $y \in \M_{\leq 1}$, any $F$, and any $n$, and show that some $y_{\alpha', F', n'}$ satisfies $p_\beta(y - y_{\alpha', F', n'}) < \frac1n$ for all $\beta \in F$.  By density of $\{x_\alpha\}$, find $x_{\alpha'}$ with $p_\beta(y - x_{\alpha'}) < \frac{1}{2n}$ for all $\beta \in F$.  Then $V_{\alpha', F, 2n}$ intersects $\M_{\leq 1}$ nontrivially (it contains $y$), so it contains an element $y_{\alpha', F, 2n}$.  Finally note that for $\beta \in F$, $p_\beta(y - y_{\alpha', F, 2n}) \leq p_\beta(y - x_{\alpha'}) + p_\beta(x_{\alpha'} - y_{\alpha', F, 2n}) < \frac1n$.

\underline{$\Dc(\M) \leq \chi_r(\M)$}: If $\M \subseteq \B(\ell^2_\kappa)$, $\M$ cannot contain a set of $> \kappa$ pairwise orthogonal projections.

\underline{$\Dc(\M) \cdot \G(\M) = \chi_r(\M)$}: From part (1) we have
$$\chi_r(\M) = \chi_r \left(\sum \nolimits_I^\oplus \M_i \right) = \sum \nolimits_I\chi_r(\M_i) \leq |I| \cdot \aleph_0 \cdot \G(\M) \leq \Dc(\M) \cdot \aleph_0 \cdot \G(\M) = \Dc(\M) \cdot \G(\M).$$
(This also uses the additivity of $\chi_r$ on direct sums, an easy fact noted as part of Theorem \ref{T:dsum} below.)  The opposite inequality follows from the preceding two underlined statements.

\underline{$\aleph_0 \cdot \chi_r(\M) = \D(\M_*)$}: Recall that $L^2(\M)$ denotes the underlying Hilbert space in a canonical left regular representation (with extra structure) called the \textit{standard form} of $\M$ (\cite{H}).  Since $L^2(\M)$ and $\M_* \simeq L^1(\M)$ are homeomorphic (\cite[Lemma 3.2]{Ra}), we have $\D(\M_*) = \D(L^2(\M)) = \aleph_0 \cdot \dim(L^2(\M))  \geq \aleph_0 \cdot \chi_r(\M)$.  On the other hand, if $\M \subseteq \B(\h)$, then
$$\D(\M_*) = \D(\B(\h)_*/\M_\perp) \leq \D(\B(\h)_*) = \aleph_0 \cdot \dim \h,$$
which suffices for the conclusion.  Here $\M_\perp$ is the preannihilator of $\M$ (the annihilator of $\M$ in $\B(\h)_*$).  The last equality is justified by identifying $\B(\h)_*$ with the trace class operators under the tracial pairing; a dense set can be obtained by choosing a basis for $\h$ and considering matrices with finitely many nonzero entries taking values in $\Q + i\Q$.

\underline{$|\M_*| = \C \cdot \chi_r(\M)^{\aleph_0}$}: This follows from the preceding underlined statement and the fact that the cardinality of any Banach space $\mathfrak{X}$ is $\D(\X)^{\aleph_0}$ (\cite[Lemma 2]{Kr}).

\underline{$|\M_*| \leq\C^{\text{gen}(\M)}$}: With $\A_0$ as in the proof of part (1), let $\A = C^*(\{x_i\})$ be the norm closure of $\A_0$.  Now $\M \simeq z\A^{**}$ for some central projection $z$ in the von Neumann algebra $\A^{**}$, and $\M_* \simeq z \A^*$.  Any linear functional on $\A$ is completely determined by its restriction to $\A_0$, so $|\M_*| = |z\A^*| \leq |\A^*| \leq \C^{|\A_0|} = \C^{\aleph_0 \cdot \text{gen}(\M)} = \C^{\text{gen}(\M)}$.
%\eqref{E:estimate2}: Represent $\M$ faithfully on a Hilbert space $\h$, and choose any $\xi \in \h$.  The $\M$-invariant subspace $\h_0 = \overline{\M \xi} = \overline{\A_0 \xi}$ has a dense set of cardinality $\leq \aleph_0 \cdot \G(\M)$, so this bounds the dimension of $\overline{\M \xi}$ (and is equal to it in case the latter is infinite).  Also $\M$ is represented normally, but not necessarily faithfully, on $\h_0$; the image of $\M$ is isomorphic to $z_\xi \M$ for some central projection $z_\xi \in \M$.  Now
%$$\kappa(\M) = \kappa(\vee_{\xi \in \h} z_\xi \M) = \sup_{\xi \in \h} \kappa(z_\xi \M) \leq \sup_{\xi \in \h} \kappa'(z_\xi \M) \leq \sup_{\xi \in \h} \chi_r(z_\xi \M) \leq \aleph_0 \cdot \G(\M). \qedhere$$
\end{proof}

%In Section  we will see that \eqref{E:estimate} totally determines which tuples of uncountable cardinals can occur as the invariants of a von Neumann algebra.

\begin{remark} \label{T:dc}
The proof of Theorem \ref{T:estimate}(1) shows that $\aleph_0 \cdot \G(\M)$ is also the density character of $\M$ or $\M_{\leq 1}$ in the $\sigma$-strong* or $\sigma$-weak topology.  %We will use this in Example \ref{T:LG}(2).
\end{remark}

\begin{example} \label{T:typeI} (Type I factors)  The representation density and decomposability number of $\B(\ell^2_\kappa)$ are easy to compute; one argument is $\kappa = \dim (\ell^2_\kappa) \geq \chi_r(\B(\ell^2_\kappa)) \geq \Dc(\B(\ell^2_\kappa)) \geq \kappa$, using \eqref{E:estimate} for the third relation and minimal projections for the fourth.  As for the $\G$ invariant, note that a type I factor cannot be written nontrivially as a direct sum, so Theorem \ref{T:estimate}(1) gives $\kappa = \chi_r(\B(\ell^2_\kappa)) \leq \aleph_0 \cdot \G(\B(\ell^2_\kappa)) \leq \aleph_0 \cdot \kappa^2$ (generating $\B(\ell^2_\kappa)$ from its matrix units).  This forces $\G(\B(\ell^2_\kappa)) = \kappa$ for $\kappa$ uncountable.  For $\kappa \leq \aleph_0$, $\B(\ell^2_\kappa)$ is singly-generated by classical results, being either finite-dimensional or properly infinite.
\end{example}

We separate out the following consequence of Theorem \ref{T:estimate} for use in Section \ref{S:main}.  It is in some sense ``known to the experts."  We could not find it fully proved in the literature, although it has been stated (\cite[bottom of p.95]{E}), and half of it (remove the modifier ``$\leq \C$") appeared as \cite[Lemma 6.5.2]{SS1}.  Its converse is also valid (see Remark \ref{T:c}(2)).

\begin{corollary} \label{T:cg}
A countably-generated von Neumann algebra $\M$ is a direct sum of $\leq \C$ separably-acting algebras.
\end{corollary}

\begin{proof}
If $\G(\M) \leq \aleph_0$, Theorem \ref{T:estimate}(1) says that $\M$ is a direct sum of $\leq \Dc(\M)$ von Neumann algebras $\M_i$, each satisfying $\chi_r(\M_i) \leq \aleph_0 \cdot \G(\M) = \aleph_0$.  Thus the $\M_i$ are separably-acting.  There are at most $\C$ of them, as $\Dc(\M) \leq \C^{\text{gen}(\M)} = \C$ by \eqref{E:estimate}.
\end{proof}

\section{An equivalent formulation of the generator problem} \label{S:main}

We start this section with some review of the relevant history.

In the very first paper on what are now called von Neumann algebras, von Neumann showed that an abelian von Neumann algebra is generated by a single self-adjoint operator (\cite[Satz 10]{vN}).  This was 1929, so Hilbert space meant $\ell^2$ (explicitly stated in the opening paragraphs), and thus the result is often stated as ``separably-acting abelian von Neumann algebras are singly-generated."  But in his proof, the first step is to note that the algebra is generated by a countable family of projections; he then gives a purely algebraic method for constructing a generator.  Since the spectral theory in the same paper shows that a singly-generated abelian von Neumann algebra is generated by a countable family of spectral projections, a countably-generated abelian von Neumann algebra is also generated by countably many projections, and von Neumann has really shown that ``countably-generated abelian von Neumann algebras are singly-generated."  (His spectral theory is developed on a separable Hilbert space, but this is not needed for the existence of spectral projections.)  Von Neumann's construction of a generator is quite intricate.  Nowadays we have an elegant one-paragraph proof that goes back at least to Rickart's 1960 book (\cite[A.2.1]{R}).

From von Neumann's result and the decomposition into real and imaginary parts, a general von Neumann algebra is singly-generated if and only if it is generated by two abelian *-subalgebras that are either countably-generated or \textit{a fortiori} separably-acting.  This seems to have been first leveraged nontrivially in Pearcy's 1962 paper \cite{P} on type I algebras.  In 1963 Suzuki and Sait\^{o} made the following observation.

\begin{lemma} \label{T:SS} $($\cite[Lemma 4]{SuS}$)$
If a von Neumann algebra is generated by countably many commuting singly-generated *-subalgebras, then it is singly-generated.
\end{lemma}

For completeness we sketch the proof.  If generators of the subalgebras are decomposed into real and imaginary parts as $x_j + i y_j$, then $W^*(\{x_j\})$ and $W^*(\{y_j\})$ are abelian and countably-generated.  By von Neumann's result each has a single self-adjoint generator, say $x$ and $y$ respectively.  Then $x + i y$ generates the original algebra.

Lemma \ref{T:SS} implies in particular that the direct sum of countably many singly-generated algebras is singly-generated (noted, for instance, in \cite[Remark, p. 451]{Saito}).  The following improvement seems to be new.

\begin{lemma} \label{T:sum}
Let $\{\M_i\}_{i \in I}$ be a set of $\leq \C$ singly-generated von Neumann algebras.  Then $\sum^\oplus \M_i$ is also singly-generated.
\end{lemma}

\begin{proof}
For each $i$, let $x_i$ be a generator for $\M_i$ with norm $\leq 1$.  Since $|I| \leq \C$, $W^*(\{e_i\}) \simeq \ell^\infty_{I}$ is a singly-generated subalgebra of the center of $\sum^\oplus \M_i$.  The commuting singly-generated algebras $W^*(\{e_i\})$ and $W^*((x_i)_i)$ together generate all of $\sum^\oplus \M_i$, which is therefore singly-generated by Lemma \ref{T:SS}.
\end{proof}

\begin{remark} \label{T:c} ${}$
\begin{enumerate}
\item Lemma \ref{T:sum} generalizes neither Lemma \ref{T:SS} nor the von Neumann result.  In particular, it does not say that an abelian von Neumann algebra generated by $\leq\C$ elements is singly-generated; that is false.  There are counterexamples in Example \ref{T:LG}(2,3) and at the end of Section \ref{S:free}.
\item Lemma \ref{T:sum} is a noncommutative analogue of the Pondiczery-Hewitt-Marczewski theorem from classical point-set topology (\cite{Po,He,M}): the Cartesian product of $\leq \C$ separable Hausdorff spaces is still separable.  In fact, this theorem and the equality $\aleph_0 \cdot \G(\M) = s\text{-}\D(\M)$ can be used to show directly that the direct sum of $\leq \C$ countably-generated von Neumann algebras is still countably-generated.  For countable generation is equivalent to $\sigma$-strong separability, and the $\sigma$-strong topology on a direct sum is the product topology.  (This can also be proved in the same way as Lemma \ref{T:sum}.)  When combined with Corollary \ref{T:cg}, this gives the following characterization: \textit{a von Neumann is countably-generated if and only if it is a direct sum of $\leq \C$ separably-acting algebras.}

    In terms of cardinal invariants, von Neumann algebras behave very much like a tractable class of topological spaces, with $\G$, $\chi_r$, and $\Dc$ substituted for density, weight, and cellularity, respectively (\cite{C}). %Although our invariants have have algebraic definitions, it is actually quite effective to view In terms of cardinal invariants, von Neumann algebras may profitably be compared to a tractable class of topological spacesT one might suggestively consider them ``quantizations" of the topological invariants   %, with  substituted .  (See  for a survey of the classical material.)  %One expects these invariants to interact for algebraic and topological reasons, but at uncountable cardinality the topological structure ``takes over" and turns out to be governed by elementary infinite combinatorics.  (This has the welcome consequence that none of our results rely on set theoretic axioms beyond ZFC.)  The generator problem asks, in a sense, whether these invariants ignore the algebraic structure altogether.
%\item We say that two operators $x,y \in \M$ are \textit{algebraically disjoint} if $1 \oplus 0 \in W^*(x \oplus y) \subseteq \M \oplus \M$.  Now if $\M$ has a generator, it obviously has $\geq \C$ distinct generators: just multiply by any nonzero scalar.  It is also not hard to see that $\M$ has $\aleph_0$ algebraically disjoint generators, by simply adding constants to a single generator so that all translated spectra are disjoint (see \cite[Lemma]{DP}).  But via Lemma \ref{T:sum} we can improve both of these: $\M$ has $\C$ algebraically disjoint generators.  For the lemma implies that $\ell^\infty_\C(\M)$ has a generator $(x_i)_i$.  When $i \neq j$, necessarily $W^*(x_i \oplus x_j) = \M_i \oplus \M_j \ni 1 \oplus 0$. [NOTE that the disjointness part (but not the lemma itself) could be gotten by $\C$ translates of a quasinilpotent generator.]
\item Lemma \ref{T:sum} is sufficient to prove the next theorem.  But the reader will guess that it can be generalized, and we do this in Theorem \ref{T:dsum} below.
\end{enumerate}
\end{remark}

\begin{theorem} \label{T:main}
Questions \ref{T:gen} and \ref{T:gen2} are equivalent: if all separably-acting von Neumann algebras are singly-generated, then all countably-generated von Neumann algebras are singly-generated.
\end{theorem}

\begin{proof}
Assume that all separably-acting von Neumann algebras are singly-generated.  Let $\M$ be countably-generated.  By Corollary \ref{T:cg}, $\M$ is a direct sum of $\leq \C$ separably-acting algebras, each singly-generated by assumption.  Then Lemma \ref{T:sum} implies that $\M$ is singly-generated.
\end{proof}

The author considers Question \ref{T:gen2} to be a natural formulation of the generator problem and closer in spirit to von Neumann's original result. Nearly all constructions involving generators have been algebraic, i.e., without reference to an underlying Hilbert space.  For example, Wogen's original proof that separably-acting properly infinite von Neumann algebras are singly-generated (\cite[Theorem 2]{Wo}) requires no change if $\M$ is only assumed to be countably-generated.  The exception is the use of direct integrals.

Recall that a von Neumann algebra is said to be \textit{approximately finite-dimensional (AFD)} if it has an increasing net of finite-dimensional *-subalgebras whose union is $\sigma$-strongly dense.

\begin{proposition} \label{T:AFD}
A countably-generated AFD von Neumann algebra $\M$ is singly-generated.
\end{proposition}

\begin{proof}
By Corollary \ref{T:cg}, $\M$ is a direct sum of $\leq \C$ separably-acting algebras, each clearly AFD.  By Lemma \ref{T:sum} it suffices to show that any separably-acting AFD algebra, say $\N$, is singly-generated.  This is known, but a little hard to pin down in the literature.  A very short argument goes by direct integral theory.  By \cite[Theorem 2]{W2}, $\N$ has a direct integral decomposition into (a.e.) AFD factors, each of which is singly-generated by \cite[Theorem 1]{SuS}.  Then their direct integral $\N$ is singly-generated (\cite[Theorem 1]{W}).
\end{proof}

\begin{remark}
%Actually \cite[Theorem 1]{SuS} handles more than factors, but the ``hyperfinite" class it covers has a somewhat more restrictive definition than in modern usage of the term.

Here is an alternate proof of the last step in Proposition \ref{T:AFD} that avoids both direct integral theory and post-1969 mathematics.  Decompose $\N$ into three summands that are type I, type $\text{II}_1$, and properly infinite.  The type $\text{II}_1$ summand is isomorphic to $\mathcal{R} \bar{\otimes} \A$, where $\A$ is abelian and $\mathcal{R}$ is the unique hyperfinite $\text{II}_1$ factor (\cite[Th\'{e}or\`{e}me 6]{Ka}).  The four commuting subalgebras $\R$, $\A$, the type I summand, and the properly infinite summand are each singly-generated by \cite[Theorem 1]{SuS}, \cite[Satz 10]{vN}, \cite{P}, and \cite[Theorem 2]{Wo}, respectively.  They generate $\N$, which is then singly-generated by Lemma \ref{T:SS}.

Suzuki and Saito wrote (\cite[p. 279]{SuS}) that single generation of $\R$ had been established in 1956 by Misonou, who apparently did not publish his proof.  But the earliest claim for this fact, also without proof, goes all the way back to Murray and von Neumann (\cite[Footnote 68]{MvN}).
\end{remark}

%Pearcy showed that the irrational rotation II_1 factor is singly generated, but he does not note that it is AFD.  Murray and von Neumann mention that it is, but omit the proof, which relies on amenability of abelian groups, developed elsewhere by vN.

We conclude this section by showing that the generator problem is also equivalent to deciding whether all \textit{finitely}-generated algebras are singly-generated, or even just all finitely-generated $\text{II}_1$ factors.  It seems possible that this reduction could be useful.

The main tool is Shen's $[0, +\infty]$-valued invariant $\GG(\cdot)$ for countably-generated tracial von Neumann algebras, which was introduced in \cite{Shen} and further developed in \cite{DSSW}.  %(Note that the existence of a trace makes the algebra $\sigma$-finite, so ``separably-acting" $\iff$ ``countably-generated" by \eqref{E:int}.)
One thinks of $\GG(\cdot)$ very roughly as a continuous version of the invariant $\G(\cdot) - 1$; it is defined to be $+\infty$ only when the algebra is not finitely-generated.  In the interest of economy we simply quote the facts we need about $\GG(\cdot)$, referring the reader to \cite[Chapter 16]{SS2} for a full treatment (including the definition).

We thank Stuart White for his suggestions on organizing this argument.

\begin{theorem} \label{T:shen}
Let $\M$ be a countably-generated $\text{II}_1$ factor.  We allow the value $\GG(\M)=+\infty$ in the (in)equalities below, with obvious interpretations.
\begin{enumerate}
\item \textsc{Bounds}. The minimal cardinality of a set of self-adjoint generators for $\M$ lies between $2\GG(\M) + 1$ and $2\GG(\M)+ 2$, inclusive $($\cite[Corollary 5.7]{DSSW}$)$.
\item \textsc{Scaling}.  For $t \in \mathbb{R}_+$, $\GG(\M_t) = \frac{\GG(\M)}{t^2}$ $($\cite[Theorem 4.5]{DSSW}$)$.  Here $\M_t$ is the usual amplification: the $\text{II}_1$ factor well-defined up to isomorphism as $p(\mathbb{M}_n \otimes \M)p$, for any $n \in \mathbb{N}$ and projection $p \in \mathbb{M}_n \otimes \M$ satisfying $\tau(p) = t/n$.
\item \textsc{Continuity}.  If $\M_1 \subseteq \M_2 \subseteq \dots \M$ are $\text{II}_1$ subfactors of $\M$ such that $\M = W^*(\cup \M_n)$ and $\GG(\M_n) = 0$ for all $n$, then $\GG(\M) = 0$ $(\sim$\cite[Theorem 5.5]{Shen}$)$.
\end{enumerate}
\end{theorem}

\begin{theorem} \label{T:finite}
The generator problem is equivalent to deciding whether all finitely-generated $\text{II}_1$ factors are singly-generated.
\end{theorem}

\begin{proof}
We first ask: \textit{What is the range of the invariant $\GG$ on countably-generated $\text{II}_1$ factors?}  In all cases where $\GG$ has been computed, the value is zero.  By Theorem \ref{T:shen}(2), it either attains all finite nonzero values or none.  If none, we claim that it does not attain the value $+\infty$ either.  For let $\M$ be an arbitrary countably-generated $\text{II}_1$ factor, and let $\{x_n\}_{n=1}^\infty$ generate $\M$.  We may choose $x_1$ so that $W^*(x_1)$ is an irreducible hyperfinite subfactor of $\M$, i.e., $W^*(x_1)' \cap \M = \mathbb{C}$ (\cite[Corollary 4.1]{Popa}).  Now for $n \in \mathbb{N}$, set $\M_n = W^*(\{x_1, \dots x_n\})$.  Each $\M_n$ is a $\text{II}_1$ factor, because any central projection has to commute with $x_1$.  And by Theorem \ref{T:shen}(1), $\GG(\M_n) \leq n - \frac12$, so the assumption that $\GG$ attains no nonzero finite values implies $\GG(\M_n)= 0$ for all $n$.  From Theorem \ref{T:shen}(3) we conclude that $\GG(\M)=0$.

Therefore the range in question is either $\{0\}$ or $[0, +\infty]$.  By Theorem \ref{T:shen}(1), this implies that the range of the $\G$ invariant on countably-generated $\text{II}_1$ factors is either $\{1\}$ or $\{1, 2, \dots, \aleph_0\}$.

Now consider the nontrivial direction in the statement to be proved.  If all finitely-generated $\text{II}_1$ factors are singly-generated, then from the previous paragraph we know that the range of $\G$ on countably-generated $\text{II}_1$ factors is $\{1\}$.  This implies a positive answer to the generator problem, as mentioned in the Introduction.
\end{proof}

\section{Cardinal invariants and direct sums} \label{S:sums}

In this section we will see that the inequality $\Dc(\M) \leq \C^{\text{gen}(\M)}$ essentially determines which pairs of cardinals arise as $(\G(\M), \Dc(\M))$.  As we have seen, these two invariants determine many others.  For cardinals $\kappa$ and $\lambda > 1$, $\log_\lambda(\kappa)$ denotes the least nonzero cardinal $\mu$ such that $\lambda^\mu \geq \kappa$.

\begin{theorem} \label{T:dsum}
Let $\{\M_i\}_{i \in I}$ be a family of von Neumann algebras.  The invariants $\chi_r$ and $\Dc$ are additive on direct sums in the sense that $\chi_r(\sum^\oplus \M_i) = \sum \chi_r(\M_i)$ and $\Dc(\sum^\oplus \M_i) = \sum \Dc(\M_i)$.  The invariant $\G$ is only subadditive and follows the formula
\begin{equation} \label{E:gen}
\G\left(\sum{}^\oplus \M_i \right) = \max\{\log_\C(|I|), \sup \G(\M_i)\}.
\end{equation}
\end{theorem}

\begin{proof}
Since each $\M_i$ can be represented faithfully on $\ell^2_{\chi_r(\M_i)}$, clearly $\sum^\oplus \M_i$ can be represented faithfully on $\oplus \ell^2_{\chi_r(\M_i)}$, which has dimension $\sum \chi_r(\M_i)$.  On the other hand, if $\sum^\oplus \M_i$ acts faithfully on $\h$, then each $\M_i = e_i (\sum^\oplus \M_i)$ acts faithfully on $e_i \h$, which therefore has dimension $\geq \chi_r(\M_i)$, entailing that $\dim \h = \sum \dim (e_i \h) \geq \sum \chi_r(\M_i)$.

The additivity of $\Dc$ is only slightly less straightforward.  For each $j \in I$ let $\{p_\alpha^j\}_{\alpha < \text{dec}(\M_j)} \subset \M_j$ be nonzero projections summing to $1_{\M_j}$.  For each $j \in I$ and $\alpha < \Dc(\M_j)$ consider the projection $(\delta_{ij} p^j_\alpha)_i \in \sum^\oplus \M_i$; this family shows that $\Dc(\sum^\oplus \M_i) \geq \sum \Dc(\M_i)$.  For the opposite inequality, let $\{q_\beta\}_{\beta \in J} \subset \sum^\oplus \M_i$ be nonzero projections summing to 1. For each $\beta$, $\sum_i e_i q_\beta = q_\beta$, so in particular $e_i q_\beta \neq 0$ for at least one $i$.  Then the nonzero projections in $\{e_i q_\beta\}_{i,\beta}$ sum to 1 and have cardinality $\geq |J|$.  Also for each $i$, the identity $\sum_\beta e_i q_\beta = e_i$ implies $|\{\beta \, | \, e_i q_\beta \neq 0\}| \leq \Dc(\M_i)$.  Finally,
$$|J| \leq |\{(i,\beta) \, | \, e_i q_\beta \neq 0\}| = \sum_i |\{\beta \, | \, e_i q_\beta \neq 0\}| \leq \sum_i \Dc(\M_i).$$

Before proving \eqref{E:gen} in generality, we handle the subcase when all $\M_i = \mathbb{C}$ and so $\sum^\oplus \M_i \simeq \ell^\infty_I$.  We simply need enough generators to separate the points of the underlying topological space $I$ (\cite[Proposition 6.1.3]{We}).  Any element of $\ell^\infty_I$ partitions the space into at most $\C$ equivalence classes as inverse images of single complex numbers.  Thus $\lambda$ elements can create up to $\C^\lambda$ equivalence classes.  Separating the points means that each equivalence class is at most a singleton, so $\lambda$ has to be large enough to satisfy $\C^\lambda \geq |I|$.

The remainder of the argument consists of establishing three inequalities.

\underline{$\G(\sum^\oplus \M_i) \ge \sup \G(\M_i)$}: If $\{x_\alpha\}$ generates $\sum^\oplus \M_i$, then $\{e_i x_\alpha\}_\alpha$ must generate $\M_i$.

\underline{$\G(\sum^\oplus \M_i) \ge \log_\C (|I|)$}: This follows readily from the computation $|I| \leq \Dc\left(\sum{}^\oplus \M_i \right) \leq \C^{\text{gen}\left(\sum^\oplus \M_i \right)}$, based on \eqref{E:estimate}.

\underline{$\G(\sum^\oplus \M_i) \leq \max\{ \sup \G(\M_i), \log_\C(|I|)\}$}: Each $\M_j$ can be generated by a set of contractions $\{y_\alpha^j\}_{\alpha < \sup \text{gen}(\M_i)}$.  For $\alpha < \sup \G(\M_i)$, set $x_\alpha = (y_\alpha^i)_i$.  Let $\{z_\beta\}_{\beta < \log_\C(|I|)}$ generate $W^*(\{e_i\}) \simeq \ell^\infty_I$, as explained at the beginning of this argument.  Then $S = \{x_\alpha\} \cup \{z_\beta\}$ is a generating set for $\sum^\oplus \M_i$ of cardinality $(\sup \G(\M_i)) + \log_\C(|I|)$.   For indices $\gamma < \min\{\log_\C(|I|), \sup_i \G(\M_i)\}$, $x_\gamma$ and $z_\gamma$ commute, since $z_\gamma$ belongs to the center: then $W^*(x_\gamma, z_\gamma)$ is singly-generated by Lemma \ref{T:SS}.  Replacing the doubletons $\{x_\gamma, z_\gamma\}$ in $S$ by singletons gives the conclusion.
\end{proof}

\begin{example} \label{T:LG} To belong to region $E$ of Figure 1, a von Neumann algebra $\M$ must have $\G(\M) > \aleph_0$ and $\Dc(\M) = \aleph_0$. The second condition is guaranteed by the existence of a faithful normal state.  Here are some examples; the main novelty probably lies in the technique of (1), the generality of (2), and the reference for (3).
\begin{enumerate}
\item Let $G$ be an ICC group, i.e., an infinite discrete group in which all non-identity conjugacy classes are infinite.  There are ICC groups of any infinite cardinality $\kappa$, for instance $\mathbb{F}_\kappa$, the free group on $\kappa$ letters.  It is well-known that when $G$ is ICC, the group von Neumann algebra $L(G)$ is a $\text{II}_1$ factor.  Using this and \eqref{E:estimate} (including the the fact, noted in its proof, that $\M_*$ and $L^2(\M)$ are always homeomorphic),
    \begin{align*}
    \aleph_0 \cdot \G(L(G)) &= \Dc(L(G)) \cdot \G(L(G)) = \D(L(G)_*) \\ &= \D(L^2(L(G))) = \dim L^2(L(G)) = \dim L^2(G) = |G|.
    \end{align*}
    In particular, $\G(L(\F_\kappa)) = \kappa$ when $\kappa > \aleph_0$.  (For the group of finite permutations of an uncountable set, the same conclusion was obtained by different methods in \cite[Proposition I.1]{Be2}.)

\item Let $\{(\M_i, \varphi_i)\}_{i \in I}$ be an infinite family of nontrivial von Neumann algebras equipped with faithful normal states, and assume that either $I$ or $\sup \G(\M_i)$ is uncountable.  Consider the tensor product $\M = \bar{\otimes} (\M_i, \varphi_i)$ with its faithful normal state $\varphi = \otimes \varphi_i$ (\cite[Section III.3.1]{B}), and identify each $\M_i$ with its image in $\M$ under the canonical inclusion.  We claim that
    \begin{equation} \label{E:ten}
    \G(\M) = |I| \cdot \sup \G(\M_i) = \sum \G(\M_i).
    \end{equation}
    The second equality follows from elementary estimates of the sum: $|I|, \: \sup \G(\M_i) \leq \sum \G(\M_i) \leq |I| \cdot \sup \G(\M_i)$, and by assumption one of $|I|$ and $\sup \G(\M_i)$ is infinite.  The first equality requires a little more assembly.  %(Some similar calculations are sprinkled through the literature; a version for separable $C^*$-algebras was noted in \cite[Lemma 2.17]{FK}.)

    It is obvious that $\G(\M) \leq \sum \G(\M_i)$, by taking the union of generating sets for the $\M_i$.

    Also observe that for each $i$, the slice map $S_i$ corresponding to the normal faithful state $\otimes_{j \neq i} \varphi_j$ on $\bar{\otimes}_{j \neq i} (\M_j, \varphi_j)$ is a normal conditional expectation from $\M$ onto $\M_i$ (\cite[Section III.2.2.6]{B}).  It follows that no $\M_i = S_i(\M)$ can have greater $\sigma$-weak density character (=$\sigma$-strong density character, see Remark \ref{T:dc}) than $\M$.  This gives $s\text{-}\D(\M) \geq \sup s\text{-}\D(\M_i)$.

    Any element of $\M_{\leq 1}$ is a $\sigma$-strong limit of finite linear combinations of finite tensors, which we may assume by Kaplansky density to belong to $\M_{\leq 1}$.  On $\M_{\leq 1}$ the $\sigma$-strong topology is generated by the norm $\|x\|_\varphi = \varphi(x^*x)^{1/2}$, so it suffices to consider limits of \textit{sequences}.  Suppose $x_n \to x$ strongly, where each $x_n$ is a finite linear combination of finite tensors as above.  Then for each $n$, $S_i(x_n)$ is a scalar for all but finitely many $i$.  Thus $S_i(x) = s\text{-}\lim S_i(x_n)$ is a scalar for all but countably many $i$.  Since $S_i$ is normal, any $\sigma$-weakly dense set must have elements that expect onto non-scalars in each $\M_i$.  It follows that when $I$ is uncountable, $s\text{-}\D(\M) \geq |I|$.  Of course this inequality is also valid when $I$ is countable.

    Putting the conclusions of the previous three paragraphs together with $s\text{-}\D(\N) = \aleph_0 \cdot \G(\N)$ from Theorem \ref{T:estimate}(1) and the second equality from \eqref{E:ten}, we get
    \begin{align}
\label{E:ten2} |I| \cdot \sup s\text{-}\D(\M_i) &\leq  s\text{-}\D(\M) = \aleph_0 \cdot \G(\M) \leq \aleph_0 \cdot \sum \G(\M_i) \\ \notag &= \aleph_0 \cdot |I| \cdot \sup \G(\M_i) = |I| \cdot \sup s\text{-}\D(\M_i).
\end{align}
    All terms of \eqref{E:ten2} are therefore equal.  By the assumption that $I$ or $\sup \G(\M_i)$ is uncountable, the $\aleph_0$ factors can be dropped, giving the first equality in \eqref{E:ten}.

    If $I$ and $\sup \G(\M_i)$ are countable, we can only give the estimate $\G(\M) \leq \sup \G(\M_i)$.  For each $i \in I$, let $\{x_j^i\}_{j < \sup \text{gen}(\M_i)}$ generate $\M_i$.  Then for each $j < \sup \G(\M_i)$, the family $\{x_j^i\}_i$ is commuting, so by Lemma \ref{T:SS}, $W^*(\{x_j^i\}_i)$ is generated by some $y^j$.  Then $\M = W^*(\{\M_i\}) = W^*(\{x_j^i\}_{i,j}) = W^*(\{y^j\}_{j < \sup \text{gen}(\M_i)})$.

    We discuss (finite) tensor products of non-$\sigma$-finite von Neumann algebras, with no reference states, in Section \ref{S:mm}.

    %Let $\{(\M_i, \tau_i)\}_{i \in I}$ be an uncountable family of tracial von Neumann algebras.  To simplify the exposition, we assume that each $\M_i$ is singly-generated and contains a projection $p_i$ with $\tau_i(p_i) = \frac12$.  The infinite tensor product $\M = \overline{\otimes} (\M_i, \tau_i)$ possesses a trace $\tau = \otimes \tau_i$ (\cite[Section III.3.1]{B}).  Let $\bar{p}_i$ be the image of $p_i$ under the inclusion $\M_i \hookrightarrow \overline{\otimes} (\M_i, \tau_i)$.  The $\sigma$-strong topology on the unit ball of $\M$ is the $L^2$-topology (\cite[Proposition III.2.2.17]{B}), so the computation
    %$$i \neq j \quad \Rightarrow \quad \|\bar{p}_i - \bar{p}_j\|_2^2 = \tau((\bar{p}_i - \bar{p}_j)^2) = \tau(\bar{p}_i + \bar{p}_j - 2 \bar{p}_i \bar{p}_j) = 1/2$$
    %implies that $s\text{-}\D(\M_{\leq 1}) \geq |I|$.  On the other hand $\M$ is generated by the set of generators for the $\M_i$ (included into $\M$); this makes $\G(\M) \leq |I|$.  We conclude from Theorem \ref{T:estimate}(1) that $\G(\M) = |I|$.

\item A tracial ultrapower of a $\text{II}_1$ factor is a $\text{II}_1$ factor (so $\sigma$-finite) that is not countably-generated.  This follows from a more general theorem proved in \cite{F} in 1956(!) -- well before ultrapower terminology was introduced in operator algebras.  See \cite[Remark 4.4 and proof of Proposition 4.3]{Popa} for the fact that a tracial ultrapower of $L^\infty[0,1]$, which has cardinality $\C$ as a quotient of $\ell^\infty(L^\infty[0,1])$, is not countably-generated.
\end{enumerate}
Note that the examples in (2) and (3) include abelian algebras.
\end{example}

Example \ref{T:LG} shows that $\G(\M)$ is not bounded by any function of $\Dc(\M)$.  One can manufacture examples with $\Dc(\M)$ strictly larger than $\G(\M)$ by exploiting the distinction between additivity and subadditivity on direct sums (Theorem \ref{T:dsum}, simple examples are $\M = \ell^\infty_{\C^\kappa}$ for any $\kappa$), but the gap is restricted by the inequality $\Dc(\M) \leq \C^{\text{gen}(\M)}$ from \eqref{E:estimate}.  This turns out to be nearly the whole story.

\begin{theorem} \label{T:range} ${}$
\begin{enumerate}
\item For any pair of cardinals $\kappa_g$ and $\kappa_d$ satisfying
\begin{equation} \label{E:bounds}
\kappa_g > \aleph_0 \qquad \text{ and } \qquad \aleph_0 \leq \kappa_d \leq \C^{\kappa_g},
\end{equation}
there is a von Neumann algebra $\M$ with $\G(\M) = \kappa_g$ and $\Dc(\M) = \kappa_d$.
\item The range of the von Neumann algebra invariant $\M \mapsto (\G(\M), \Dc(\M))$ is the union of the following three sets:
\begin{enumerate}
\item all values allowed by \eqref{E:bounds};
\item $\{(1, \kappa) \mid 1 \leq \kappa \leq \C\}$;
\item either $\varnothing$, or $[2, \aleph_0] \times [\aleph_0,\C]$.
\end{enumerate}
(The generator problem asks whether the third set is $\varnothing$.)
\end{enumerate}
\end{theorem}

\begin{proof}
(1): We claim that $\M = \ell^\infty_{\kappa_d} (L(\F_{\kappa_g}))$ works.  By Theorem \ref{T:dsum} and Example \ref{T:LG}(1), we compute
$$\Dc(\ell^\infty_{\kappa_d} (L(\F_{\kappa_g}))) = \kappa_d \cdot \Dc(L(\F_{\kappa_g})) = \kappa_d \cdot \aleph_0 = \kappa_d, \qquad\G(\ell^\infty_{\kappa_d} (L(\F_{\kappa_g}))) = \max\{\log_\C(\kappa_d ), \kappa_g\} = \kappa_g.$$

%Given $\kappa_g$ and $\kappa_d$ as above, we claim that  For one, $\Dc(\M)$ = $\kappa_d \cdot \Dc(L(\F_{\kappa_g})) = \kappa_d \cdot \aleph_0 = \kappa_d$.  Also $\G(\M) \geq \kappa_g$, since the generators' entries in any one coordinate must generate $L(\F_{\kappa_g})$.  [MAYBE MAKE THESE ELEMENTARY OBSERVATIONS EARLIER.  Clean up this proof.]  As for the opposite inequality, let $\{x_\alpha\}_{\alpha < \kappa_g}$ generate $L(\F_{\kappa_g})$.  Let $\lambda$ be the least cardinal such that $\C^\lambda = \kappa_d$; by ?? the center of $\M$, which is isomorphic to $\ell^\infty_{\kappa_d}$, can be generated by some set $\{y_\beta\}_{\beta < \lambda}$.  By assumption $\lambda \leq \kappa_g$, so $\{x_\alpha\} \cup \{y_\beta\}$ has cardinality $\kappa_g$ and generates $\M$.

(2): It follows from part (1) and the algebras $\ell^\infty_\kappa$ that the values in (a) and (b) are attained.  Also these are the only possibilities when $\G(\M)$ is 1 or uncountable, because of the inequality $\Dc(\M) \leq \C^{\text{gen}(\M)}$ from \eqref{E:estimate}, and the fact that $\G(\M) > 1$ implies that $\M$ is infinite-dimensional and so $\Dc(\M) \geq \aleph_0$.  %The remaining assertion, that $S$ is one of the two given sets, will be proved in Section \ref{S:finite}.

We showed in the proof of Theorem \ref{T:finite} that the range of the $\G$ invariant on countably-generated $\text{II}_1$ factors is either $\{1,2, \dots , \aleph_0\}$ or $\{1\}$, and the latter case implies that $\G$ takes no values in $[2,\aleph_0]$ on any algebra.  In the former case, choose any $(\lambda, \mu) \in [2, \aleph_0] \times [\aleph_0, \C]$.  By assumption there is a $\text{II}_1$ factor $\M$ with $\G(\M) = \lambda$; then $\G(\ell^\infty_\mu(\M)) = \lambda$ and $\Dc(\ell^\infty_\mu(\M)) = \mu$ by Theorem \ref{T:dsum}.
\end{proof}

%We show next that the rest of the range has the form $S \times [\aleph_0, \C]$ for some $S \subseteq [2, \aleph_0]$.
%It remains to see that if there is $\M$ with $\G(\M) = \lambda$ for some $2 \le \lambda \leq \aleph_0$, then for every $\mu \in [\aleph_0, \C]$ there is $\N$ with $\G(\N) = \lambda$ and $\Dc(\N) = \mu$.  We first explain how to decrease the decomposability to $\aleph_0$; then we explain how to increase it to any value $\leq \C$.

%Suppose $\G(\M) = \lambda \in [2, \aleph_0]$, with $\Dc(\M)$ anywhere in $(\aleph_0, \C]$.  By Corollary \ref{T:cg}, $\M$ can be written as $\sum^\oplus_{i \in I} \M_i$, with $|I| \leq \C$ and each $\M_i$ $\sigma$-finite.  Since $\log_\C(|I|) = 1$, $\lambda = \sup \G(\M_i)$ by Theorem \ref{T:dsum}.  Now if the supremum is achieved, one of the $\M_i$ has $\G(\M_i) = \lambda$, and we can take this $\M_i$ as $\N$.  For necessarily $\Dc(\M_i) = \aleph_0$, as $\M_i$ is $\sigma$-finite but not finite-dimensional (otherwise it would be singly-generated).  If the supremum is not achieved, $\lambda$ must be $\aleph_0$, and one can find a countable sequence of summands $\M_{i_j}$ with $\sup_j \G(\M_{i_j}) = \aleph_0$.  Taking $\N = \sum^\oplus_j \M_{i_j}$, $\Dc(\N) = \aleph_0 = \G(\N)$ from Theorem \ref{T:dsum}.

%Finally suppose $\G(\M) = \lambda \in [2,\aleph_0]$ and $\Dc(\M) = \aleph_0$, with $\mu$ anywhere in $(\aleph_0, \C]$.  Then $\Dc(\ell^\infty_\mu(\M)) = \mu$ and $\G(\ell^\infty_\mu(\M)) = \lambda$ by Theorem \ref{T:dsum}.

\begin{remark}
Note that in Theorem \ref{T:range}, $L(\F_{\kappa_g})$ could be replaced with any algebra with the same $\G$ and $\Dc$ invariants, even an abelian one (Example \ref{T:LG}(2)).  So if the generator problem has an affirmative answer, the entire range of the invariant $(\G(\M), \Dc(\M))$ is achieved on abelian algebras.
\end{remark}

\begin{remark}
As stated, the converse to Theorem \ref{T:main} is trivial.  However, looking at Figure 1, Theorem \ref{T:main} could be phrased, ``If $D$ is empty then $A$ is empty."  This statement's converse follows from the last part of Theorem \ref{T:range}.  An algebra $\M$ lying in region $D$ would have $\G(\M) \in [2, \aleph_0]$ and $\Dc(\M) = \aleph_0$; then $\G(\ell^\infty_\C(\M)) = \G(\M)$ and $\Dc(\ell^\infty_\C(\M)) = \C$ by Theorem \ref{T:dsum}, making $\ell^\infty_\C(\M)$ an element of region $A$.
\end{remark}

%\begin{proposition}
%If there is a von Neumann algebra that is separably-acting but not singly-generated, then there is a countably-generated von Neumann algebra that is neither separably-acting nor singly-generated.
%\end{proposition}

%\begin{proof}
%Assuming the hypothesis, let $\M$ be separably-acting but not singly-generated.  Still $\M$ is at least countably-generated, say by $\{x_j\}_{j \geq 1}$.  Consider $\N = \ell^\infty_\C(\M)$.

%First, $\N$ is still countably-generated.  For each $j \geq 1$ set $X_j$ to be the constant sequence $(x_j)_\alpha$.  With $X_0$ a generator for the copy of $\ell^\infty_\C$ generated by the coordinate projections, $\N = W^*(\{X_j\}_{j \geq 0})$.

%Second, $\N_* \simeq \ell^1_\C(\M_*)$, which is obviously nonseparable.

%Third, $\N$ is not singly-generated.  For if it were generated by $(y_\alpha)_\alpha \in \N$, then each $y_\alpha$ would be a generator for $\M$.
%\end{proof}

\section{Cardinal invariants and the center} \label{S:center}

In the previous section we built our examples satisfying $\Dc(\M) > \G(\M)$ as direct sums.  This is unavoidable, as the first part of the next proposition shows.

\begin{proposition} \label{T:center} Let $\M$ be a von Neumann algebra.
\begin{enumerate}
\item If $\Dc(\M) > \aleph_0 \cdot \G(\M)$, then $\Dc(\M) = \Dc(\z(\M))$.
\item $\aleph_0 \cdot \chi_r(\M) = \aleph_0 \cdot \Dc(\z(\M)) \cdot \G(\M)$.
\item If $\M$ has $\sigma$-finite center, then $\chi_r(\M) \leq \aleph_0 \cdot \G(\M)$, with equality when $\M$ is infinite-dimensional.
\item (Strengthening of Theorem \ref{T:estimate}(1)) $\M$ can be written as a direct sum in which each summand $\M_i$ is either some $\mathbb{M}_n$ or satisfies $\chi_r(\M_i) = \aleph_0 \cdot \G(\M_i)$.
\end{enumerate}
\end{proposition}

\begin{proof}
(1): Assume $\Dc(\M) > \aleph_0 \cdot \G(\M)$ and let $\M = \sum^\oplus \M_i$ be as in Theorem \ref{T:estimate}(1).  Compute
$$\chi_r(\M) = \sum \chi_r(\M_i) \leq \Dc(\z(\M)) \cdot \aleph_0 \cdot \G(\M) \leq \Dc(\M) \cdot \aleph_0 \cdot \G(\M) = \Dc(\M) = \chi_r(\M),$$
using additivity of $\chi_r$ for the first step and \eqref{E:estimate} for the fifth.  Then
$$\aleph_0 \cdot \G(\M) < \Dc(\M) = \Dc(\z(\M)) \cdot \aleph_0 \cdot \G(\M) = \max\{\Dc(\z(\M)), \aleph_0 \cdot \G(\M)\}$$
implies that the maximum on the right is $\Dc(\z(\M))$.

(2): By \eqref{E:estimate} we have
$$\aleph_0 \cdot \chi_r(\M) = \aleph_0 \cdot \Dc(\M) \cdot \G(\M),$$
and by (1) either the right-hand side is $\aleph_0 \cdot \G(\M)$ or $\Dc(\M) = \Dc(\z(\M))$.

%We prove the contrapositive.  If $\chi_r(\M) > \aleph_0 \cdot \G(\M)$, then by \eqref{E:estimate} and part (1), $ \chi_r(\M) = \Dc(\M) = \Dc(\z(\M))$.

(3): Follows directly from part (2).

(4): Follows from part (3) by writing $\M$ as a direct sum of its matricial summands and arbitrary other summands with $\sigma$-finite center.
%By part (3) and \eqref{E:estimate}, an infinite-dimensional summand $\M_i$ satisfies
%$$\chi_r(\M_i) \leq \aleph_0 \cdot \G(\M_i) \leq \aleph_0 \cdot \chi_r(\M_i) = \chi_r(\M_i). \qedhere$$
\end{proof}

\begin{remark} \label{T:dixmier}
We cited Dixmier's book (\cite[Exercice I.7.3bc]{D}) for the classical fact \eqref{E:int} that ``separably-acting" is the same as ``countably-generated and $\sigma$-finite," then gave the equation $\chi_r(\M) = \G(\M) \cdot \Dc(\M)$ as a generalization.  The same exercises in Dixmier also show that ``separably-acting" is equivalent to ``countably-generated and having $\sigma$-finite center," which is generalized by Proposition \ref{T:center}(2).
\end{remark}

Next we consider modified invariants that ignore the size of the center, at least in terms of decomposability.  If $\M \mapsto F(\M)$ is any cardinal invariant, its ``localization" is
$$F'(\M) = \min\{\kappa \mid \M \text{ can be written as a direct sum of algebras $\{\M_\alpha\}$ with $F(\M_\alpha) \leq \kappa$ for all $\alpha$}\}.$$

\begin{lemma}
Assume that a cardinal invariant $F$ has the regularity property $F(\M) \leq F(\M \oplus \N)$ for arbitrary $\M$ and $\N$, as all invariants in this paper do.  Then for any decomposition $\M = \sum^\oplus \M_i$,
\begin{equation} \label{E:F'}
F'(\M) = \sup \nolimits_i F'(\M_i).
\end{equation}
\end{lemma}

\begin{proof}
For each $i$ let $\M_i = \sum^\oplus_{j \in J_i} \M_j^i$ be such that $F'(\M_i) = \sup_{j \in J_i} F(\M_j^i)$.  Then
$$\sup_i F'(\M_i) = \sup_i \sup_{j \in J_i} F(\M_j^i) \geq F'(\M),$$
since $\M$ is the direct sum of all the $\M_j^i$.  In the other direction, let $\M = \sum^\oplus \M_\alpha$ be such that $F'(\M) = \sup_\alpha F(\M_\alpha)$.  For any $i_0$ we have
$$F'(\M) = \sup_\alpha F(\M_\alpha) \geq \sup_{ \substack{\alpha, i \\ \M_\alpha \cap \M_i \neq 0}} F(\M_\alpha \cap \M_i) \geq \sup_{ \substack{\alpha \\ \M_\alpha \cap \M_{i_0} \neq 0}} F(\M_\alpha \cap \M_{i_0}) \geq F'(\M_{i_0}).$$
This implies $F'(\M) \geq \sup_i F'(\M_i)$.
\end{proof}

Here are some applications of invariants of this type.

\smallskip

\textbf{1.} The smallness criterion $\Dc'(\M) \leq \aleph_0$ means that $\M$ is a direct sum of $\sigma$-finite algebras.  It has implications for dimension theory (\cite[Proposition 3.8]{S}, where $\Dc'(\M)$ is denoted ``$\kappa_\M$").

\smallskip

%\textbf{2.} Lemma \ref{T:sum} and \eqref{E:estimate} show $[\Dc(\M) \leq \C \text{ and } \: \G'(\M) = 1] \Leftrightarrow \G(\M) = 1$.

%\smallskip

\textbf{2.} The main content of Theorem \ref{T:estimate}(1) is the inequality
\begin{equation} \label{E:center}
\chi_r'(\M) \leq \aleph_0 \cdot \G(\M).
\end{equation}
Here is an improvement.

\begin{proposition} \label{T:same}
For a von Neumann algebra $\M$, we have
$$\aleph_0 \cdot \chi_r'(\M) = \aleph_0 \cdot \G'(\M).$$
Thus the invariants $\G'(\M)$ and $\chi'_r(\M)$ only differ when $\G'(\M)$ is finite and $\M$ is not atomic abelian.
\end{proposition}

\begin{proof}
Let $\M = \sum^\oplus \M_i$ be a decomposition such that $\G'(\M) = \sup \G(\M_i)$.  Compute
$$\aleph_0 \cdot \chi'_r(\M) = \aleph_0 \cdot \sup \chi'_r(\M_i) \leq \aleph_0 \cdot \sup \G(\M_i) = \aleph_0 \cdot \G'(\M),$$
where the first two relations are justified by \eqref{E:F'} and \eqref{E:center}, respectively.  The opposite inequality follows from the general fact $\G(\N) \leq \chi_r(\N)$ from \eqref{E:estimate}.

The necessary observation for the second sentence is $\G'(\M) > 1 \Rightarrow \chi'_r(\M) \geq \aleph_0$ (because a summand that is not singly-generated must be infinite-dimensional and so has infinite representation density).
\end{proof}

Proposition \ref{T:same} generalizes a result of Kehlet (\cite[Proposition 1]{K}), where it is shown that $\G'(\M) \leq \aleph_0 \iff \chi'_r(\M) \leq \aleph_0$.

\smallskip

\textbf{3.} We can also generalize \cite[Proposition 2]{K}, which says that if $\{\M_n\}$ is a countable set of von Neumann algebras acting on a common Hilbert space, and $\chi'_r(\M_n) \leq \aleph_0$ for each $n$, then $\chi'_r(W^*(\{\M_n\})) \leq \aleph_0$ too.  The broader fact is that for any family $\{\M_i\}_{i \in I}$ on a common Hilbert space,
$$\chi'_r(W^*(\{\M_i\})) \leq |I| \cdot \aleph_0 \cdot \sup \chi'_r(\M_i).$$
Here is the idea, not much different from \cite{K} or the proof of Theorem \ref{T:estimate}(1) above.  For any nonzero vector $\xi$ and index $i$, let $\M_i$ be decomposed into summands that are each generated by $\leq \G'(\M_i)$ elements.  All but countably many summands of $\M_i$ annihilate $\xi$, so all but $\leq \aleph_0 \cdot \G'(\M_i)$ generators of $\M_i$ annihilate $\xi$.  At most $|I| \cdot \aleph_0 \cdot \sup \G'(\M_i)$ generators of $\M$ fail to annihilate $\xi$, so the invariant subspace $\overline{\M \xi}$ has a dense set of cardinality $\leq |I| \cdot \aleph_0 \cdot \sup \G'(\M_i)$ ($=|I| \cdot \aleph_0 \cdot \sup \chi'_r(\M_i)$ by Proposition \ref{T:same}).  The rest of the argument is the same as for Theorem \ref{T:estimate}(1).

%The hypothesis implies that the right-hand side is $\Dc(\M) = \chi_r(\M)$
%\begin{equation} \label{E:2}
%\chi_r(\M) = \Dc(\M) \cdot \G(\M) = \Dc(\M) > \aleph_0 \cdot \G(\M) \geq \chi'_r(\M),
%\end{equation}
%the four relations justified by \eqref{E:estimate}, hypothesis, hypothesis, and \eqref{E:center}, respectively.  From \eqref{E:1} and \eqref{E:2} we conclude that
%\begin{equation} \label{E:3}
%\Dc(\z(\M)) \cdot \chi'_r(\M) \geq \chi_r(\M) > \chi'_r(\M).
%\end{equation}
%Since the middle term of \eqref{E:3} is infinite, so is the left term, which is therefore actually a maximum.  And $\max\{\Dc(\z(\M)), \chi'_r(\M)\} > \chi'_r(\M)$ entails that the maximum is $\Dc(\z(\M))$, so that $\Dc(\z(\M)) \geq \chi_r(\M)$.  Finally
%$$\Dc(\z(\M)) \geq \chi_r(\M) = \Dc(\M) \geq \Dc(\z(\M))$$
%gives the desired conclusion.
%\end{proof}

%no converse: dec(M) uncountable and equal to dec(z(M)) does not imply greater than gen(M) -- direct sum of big II1 factors

\section{Monotonicity and multiplicativity of the invariant $\G(\M)$} \label{S:mm}

We say that a cardinal invariant $F$ is \textit{monotone} if $\N \subseteq \M$ entails $F(\N) \leq F(\M)$.  (We do not require that inclusions be unital.)  It is obvious that $\Dc$ and $\chi_r$ are monotone.  What about $\G$?

\begin{theorem} \label{T:monotone} ${}$
\begin{enumerate}
\item If there exists an inclusion $\N \subseteq \M$ such that $\G(\N) > \G(\M)$, then $\M$ is finitely-generated and $\N$ is countably-generated.
\item The invariant $s$-$\D$ is monotone.
\item The generator problem is equivalent to deciding whether $\G$ is monotone.
\end{enumerate}
\end{theorem}

\begin{proof}
(1): Suppose $\N \subseteq \M$ and $\G(\N) > \G(\M)$.  Find nonzero $\sigma$-finite projections $\{e_i\}_{i \in I} \subset \z(\M)$ that sum to 1.  Writing $\M_i = e_i \M$, we have $\M = \sum^\oplus \M_i$ with $\Dc(\z(\M_i)) \leq \aleph_0$.

For each $i$ the algebra $e_i \N$ is isomorphic to a direct summand $z_i \N$ of $\N$.  Since the inclusion $\N \hookrightarrow \M$ is faithful, $\vee z_i = 1_\N$.  Well-order the indices and set $y_i = z_i - \vee_{j < i} z_j$, so that $\sum y_i = 1_\N$.  Let $I' = \{i \in I \mid y_i \neq 0\}$.  Writing $\N_i = y_i \N$, we have $\N \simeq \sum^\oplus_{I'} \N_i$, and for each $i \in I'$ the embedding $z_i \N \hookrightarrow \M_i$ carries $\N_i$ isomorphically onto a subalgebra of $\M_i$.

By Theorem \ref{T:dsum} and hypothesis,
$$\max\{\log_\C(|I|), \sup\nolimits_{i \in I} \G(\M_i)\} = \G(\M) < \G(\N) = \max\{\log_\C(|I'|), \sup\nolimits_{i \in I'} \G(\N_i)\}.$$
Since $I' \subseteq I$, the right-hand side must be $\sup_{i \in I'} \G(\N_i)$.  The inequality $\G(\N_i) > \G(\M)$ must then happen for some $i = i_0$; we show that this entails countability of $\G(\N_{i_0})$ and finiteness of $\G(\M)$.  Since $\N_{i_0} \hookrightarrow \M_{i_0}$, we get
\begin{equation} \label{E:summand}
\G(\M_{i_0}) \leq \sup\nolimits_{i \in I} \G(\M_i) \leq \G(\M) < \G(\N_{i_0}) \leq \chi_r(\N_{i_0}) \leq \chi_r(\M_{i_0}) \leq \aleph_0 \cdot \G(\M_{i_0}),
\end{equation}
using Proposition \ref{T:center}(3) for the last relation.  Comparing the end terms, $\G(\M_{i_0})$ must be finite, making $\G(\N_{i_0})$ countable and $\G(\M)$ finite.  We have shown that $\G(\N_i) > \G(\M) \Rightarrow \G(\N_i) \leq \aleph_0$, so $\G(\N) = \sup \G(\N_i) \leq \aleph_0$.

(2): Part (1) guarantees that for any inclusion $\N \subseteq \M$, $\aleph_0 \cdot \G(\N) \leq \aleph_0 \cdot \G(\M)$.  Then the conclusion follows from Theorem \ref{T:estimate}(1).

(3): If $\G$ is monotone, then $\M \subseteq \B(\ell^2) \Rightarrow \G(\M) \leq \G(\B(\ell^2)) = 1$, giving a ``yes" answer to Question \ref{T:gen}.  A ``yes" answer to Question \ref{T:gen} entails a ``yes" answer to Question \ref{T:gen2} by Theorem \ref{T:main}.  Finally, a ``yes" answer to Question \ref{T:gen2} implies that $\G$ must be monotone by part (1).
\end{proof}

\begin{corollary} \label{T:mult}
Let $\M$ and $\N$ be von Neumann algebras.
\begin{enumerate}
\item $\G(\M \bar{\otimes} \N) \leq \max\{\G(\M), \G(\N)\}$.
\item If at least one of $\M$ and $\N$ is not countably-generated, then
\begin{equation} \label{E:mult}
\G(\M \bar{\otimes} \N) = \G(\M) \cdot \G(\N).
\end{equation}
\item The generator problem is equivalent to deciding whether \eqref{E:mult} is universally valid, i.e., whether $\G$ is multiplicative on tensor products.
\end{enumerate}
\end{corollary}

\begin{proof}
(1): Same argument as the second-to-last paragraph of Example \ref{T:LG}.  %Let $\{x_\alpha\}_{\alpha < \text{gen}(\M)}$ and $\{y_\beta\}_{\beta < \text{gen}(\N)}$ be minimal generating sets for $\M$ and $\N$, respectively.  Their union generates $\M \bar{\otimes} \N$.  Just as in the end of the proof of Theorem \ref{T:dsum}, for indices $\gamma < \min\{\G(\M), \G(\N)\}$ one can replace $\{x_\gamma, y_\gamma\}$ in the generating set by a single generator for $W^*(x_\gamma, y_\gamma)$ (which exists by Lemma \ref{T:SS}).

(2): Use part (1) and Theorem \ref{T:monotone}(1), noting $\M$ and $\N$ are subalgebras of $\M \bar{\otimes} \N$.  %Without loss of generality assume $\G(\N) \leq \G(\M) > \aleph_0$.  Since $\M \subseteq \M \bar{\otimes} \N$, by part (1) and Theorem \ref{T:monotone}(1),
%$$\G(\M \bar{\otimes} \N) \leq \max\{\G(\M), \G(\N)\} = \G(\M) \cdot \G(\N) = \G(\M) \leq \G(\M \bar{\otimes} \N).$$

(3): A ``yes" answer to Question \ref{T:gen2} makes both sides of \eqref{E:mult} equal 1 whenever $\M$ and $\N$ (so also $\M \bar{\otimes} \N$) are countably-generated.  A ``no" answer to Question \ref{T:gen2} implies the existence of $\M$ with $\G(\M) \in (1, \aleph_0]$.  Since $\M \otimes \B(\ell^2)$ is countably-generated and properly infinite,
\begin{equation} \label{E:ctrex}
\G(\M \bar{\otimes} \B(\ell^2)) = 1 < \G(\M) = \G(\M) \cdot \G(\B(\ell^2)). \qedhere
\end{equation}
\end{proof}

Tensoring with $\B(\ell^2_\kappa)$ can either increase or decrease $\G(\M)$, depending on $\kappa$ and the answer to the generator problem.  For uncountable $\kappa$, by Example \ref{T:typeI} and Corollary \ref{T:mult}(2) the action is nondecreasing, and even strictly increasing when $\kappa > \G(\M)$.  But for countable $\kappa$, the action is nonincreasing by Corollary \ref{T:mult}(1).  In fact, if the generator problem has a negative answer, the decreasing effect gets stronger as $\kappa$ increases, until at $\kappa = \aleph_0$ all countably-generated tensor products are singly-generated.  This is illustrated by two very nice results from a neglected 1972 paper of Behncke.

\begin{theorem} \label{T:behncke} $($\cite[Lemma 2, Theorem 1 and following remark]{Be}$)$
Let $\M$ and $\N$ be separably-acting von Neumann algebras.
\begin{enumerate}
\item If $\M$ is generated by $n$ self-adjoint operators, then $\mathbb{M}_k \otimes \M$ can be generated by $m \geq 2$ self-adjoint operators as long as $m-1 \geq \frac{n-1}{k^2}$.
\item If $\M$ and $\N$ lack finite type I summands, then $\M \bar{\otimes} \N$ is singly-generated.
\end{enumerate}
\end{theorem}
From inspection of its proof, this theorem remains valid if ``separably-acting" is replaced by ``countably-generated."

The fact that $\G$ is nonincreasing under tensoring with a matrix algebra is a hallowed trick in the history of the generator problem, evolving quickly from its inception in Pearcy's 1963 paper \cite{Pe2}.  Based on a fairly thorough survey of the literature, the author concluded that Theorem \ref{T:behncke}(1) is the sharpest result of this type.  It is essentially the strongest possible implication that is compatible with the contingency that $L(\F_m)$ is not generated by fewer than $m$ self-adjoint operators, because of (and exactly matching) Voiculescu's isomorphism $L(\F_m) \simeq \mathbb{M}_k \otimes L(\F_{1 + k^2(m-1)})$ for $1 \leq k < \aleph_0$ and $1 < m \leq \aleph_0$ (\cite[Theorem 3.3(b)]{V}).  Notice that Shen's invariant $\GG(\cdot)$ scales similarly under tensoring with matrix algebras (Theorem \ref{T:shen}(2)).

Special cases of Theorem \ref{T:behncke}(2) include properly infinite $\M$ (since then $\M \simeq \M \otimes \B(\ell^2)$) and tensor products of $\text{II}_1$ factors (later reobtained as \cite[Theorem 6.2(c)]{GP}).

\smallskip

For completeness we observe that $\chi_r$ and $\Dc$ are multiplicative on tensor products.

\begin{proposition} \label{T:mult2}
If $\M$ and $\N$ are von Neumann algebras, then
$$\chi_r(\M \bar{\otimes} \N) = \chi_r(\M) \cdot \chi_r(\N) \quad \text{ and } \quad \Dc(\M \bar{\otimes} \N) = \Dc(\M) \cdot \Dc(\N).$$
\end{proposition}

\begin{proof}
This is straightforward when both algebras are finite-dimensional, so assume that at least one is infinite-dimensional.

The inclusions $\M, \N \subseteq \M \bar{\otimes} \N$ give the relation $\chi_r(\M) \cdot \chi_r(\N) = \max\{ \chi_r(\M), \chi_r(\N) \} \leq \chi_r(\M \bar{\otimes} \N)$ by monotonicity. For the opposite inequality just note that if $\M$ acts faithfully on $\h$ and $\N$ acts faithfully on $\kH$, then $\M \bar{\otimes} \N$ acts faithfully on $\h \otimes \kH$ by construction (\cite[III.1.5.4]{B}).

For $\Dc$, let $\{p_\alpha\}_{\alpha < \text{dec}(\M)} \subset \M$ and $\{q_\beta\}_{\beta < \text{dec}(\N)} \subset \N$ be families of nonzero projections adding to 1.  By \cite[Theorem 2.6(i)]{HN} we may assume that all these projections are $\sigma$-finite.  Then the $\Dc(\M) \cdot \Dc(\N)$ projections $\{p_\alpha \otimes q_\beta\}_{\alpha < \text{dec}(\M), \beta < \text{dec}(\N)}$ are an infinite family of $\sigma$-finite nonzero projections adding to 1 as well, so again by \cite[Theorem 2.6(i)]{HN}, $\Dc(\M \bar{\otimes} \N) = \Dc(\M) \cdot \Dc(\N)$.  (Comments: (1) The result in \cite[Theorem 2.6(i)]{HN} refers to cyclic projections instead of $\sigma$-finite projections.  The former concept depends on a choice of representation and the latter does not, but in a suitable (say, standard) representation they agree.  (2) To see that the tensor product of $\sigma$-finite projections is $\sigma$-finite, note that $\sigma$-finiteness is equivalent to being the support of a normal state.  If $\varphi$ is supported on $p_\alpha$ and $\psi$ is supported on $q_\beta$, then $\varphi \otimes \psi$ is supported on $p_\alpha \otimes q_\beta$ (\cite[Proposition III.2.2.29]{B}).)
\end{proof}

\section{Some remarks on cardinal invariants for double duals of $C^*$-algebras} \label{S:dd}

\subsection{Double duals of full $C^*$-algebras of free groups} \label{S:free}
If $\G(\M) = \kappa$, then $\M$ is generated by $\leq 2\kappa$ unitaries and is therefore a summand of $C^*(\mathbb{F}_{2\kappa})^{**}$.  Thus $C^*(\mathbb{F}_{2\kappa})^{**}$ is the ``largest" von Neumann algebra generated by $\kappa$ elements.  Note that one can construct a one-dimensional representation of $\mathbb{F}_{2\kappa}$ by sending the $2\kappa$ generators to arbitrary unit scalars.  This produces $\C^{2\kappa} = \C^\kappa$ distinct 1-dimensional summands in $C^*(\mathbb{F}_{2\kappa})^{**}$, which therefore has decomposability number $\geq \C^\kappa$.  On the other hand $C^*(\mathbb{F}_{2\kappa})^{**}$ is visibly generated by $\kappa$ elements.  The general relation $\Dc(\M) \leq \C^{\text{gen}(\M)}$ from \eqref{E:estimate} then forces $\Dc(C^*(\mathbb{F}_{2\kappa})^{**}) = \C^\kappa$.  While this argument does not show that $\G(C^*(\mathbb{F}_{2\kappa})^{**})$ equals $\kappa$, it is not larger, and $\C^{\text{gen}(C^*(\mathbb{F}_{2\kappa})^{**})} = \C^\kappa$.  In other words
\begin{equation} \label{E:free}
\log_\C(\C^\kappa) \leq \G(C^*(\F_{2\kappa})^{**}) \leq \kappa.
\end{equation}
By the same analysis, \eqref{E:free} also applies if $\F_{2\kappa}$ is replaced with $\F^\text{ab}_{2\kappa}$, the free abelian group on $2\kappa$ generators.  In particular $1 < \log_\C(\C^\C) \leq \G(C^*(\F^\text{ab}_\C)^{**}) \leq \C$; this phenomenon was mentioned in Remark \ref{T:c}(1).  (Incidentally, each of the relations $\log_\C(\C^\C) = \C$ and $\log_\C(\C^\C) < \C$ is consistent with ZFC.  The author thanks Ilijas Farah for explaining this to him.)

\subsection{Relations to the work of Hu and Neufang} \label{S:HN}
Hu and Neufang proved many results about $\Dc(\M)$ in \cite{HN, Hu, N}, especially for von Neumann algebras that are second duals and/or associated to locally compact groups.  As remarked earlier, the intersection between these papers and the present one mostly concerns \eqref{E:estimate}.  Here is something interesting that follows from the union: for $G$ any infinite locally compact group, $\max\{\Dc(L(G)), \Dc(L^\infty(G))\} = \chi_r(L(G))$.  (Both quantities equal $\dim L^2(G)$, by \cite[Proof of Lemma 7.6]{HN} and the proof of ``$\aleph_0 \cdot \chi_r(\M) = \D(\M_*)$" in Theorem \ref{T:estimate}(2).)

In the rest of this section we apply our results to two questions raised by Hu and Neufang.

In \cite[Remark 6.7(ii)]{HN} they ask whether $\Dc(\A^{**}) = |\A^*|$ for every infinite-dimensional unital commutative $C^*$-algebra $\A$.  The answer to this question is no.  Let $I$ be an infinite set whose cardinality satisfies $|I| < |I|^{\aleph_0}$; for example $|I|$ could be $\aleph_0$ or $\aleph_\omega$ (the latter by K\"{o}nig's theorem, see \cite[Corollary 5.14]{J}.)  Let $\A$ be the unitization of $c_0(I)$, so that $\A^{**} \simeq \ell^\infty_I$.  Then from \eqref{E:estimate},
$$|\A^*| = |(\A^{**})_*| = \C \cdot  \chi_r (\ell^\infty_I )^{\aleph_0} = |I|^{\aleph_0} > |I| = \Dc(\ell^\infty_I) = \Dc(\A^{**}).$$

In \cite[Remark 6.7(iii)]{HN} they ask for which infinite-dimensional $C^*$-algebras $\A$ one has
\begin{equation} \label{E:HN}
\Dc(\A^{**}) = \D(\A^*).
\end{equation}
(One always has the relation ``$\leq$," by \cite[Corollary 2.7]{HN} or \eqref{E:estimate}.)  Since $\D(\A^*) = \D((\A^{**})_*) = \chi_r(\A^{**})$ also by \eqref{E:estimate}, condition \eqref{E:HN} says that $\A^{**}$ is \textit{maximally decomposable} (\cite[Definition 3.1]{Hu}), the main concept of Hu's paper \cite{Hu}.  Hu and Neufang show that \eqref{E:HN} holds for many classes of $C^*$-algebras associated to infinite locally compact groups.  %They also point out that \eqref{E:HN} is clearly true when $\A^*$ is separable, as infinite-dimensionality of $\A^{**} \supset \A$ implies $\Dc(\A^{**}) \geq \aleph_0$.
\textit{Does \eqref{E:HN} hold for all infinite-dimensional $C^*$-algebras?}  We do not know, but at least it is widely enjoyed.

\begin{proposition} \label{T:last}
If an infinite-dimensional $C^*$-algebra $\A$ is either type I, or generated by $\leq \C$ elements, then \eqref{E:HN} holds.
\end{proposition}

\begin{proof}
We will work with a reformulation of \eqref{E:HN}.  Since $\Dc(\A^{**}) \cdot \G(\A^{**}) = \D(\A^*)$ by \eqref{E:estimate}, \eqref{E:HN} is equivalent to $\Dc(\A^{**}) = \Dc(\A^{**}) \cdot \G(\A^{**})$, which is in turn equivalent to
\begin{equation} \label{E:HN'}
\G(\A^{**}) \leq \Dc(\A^{**}).
\end{equation}

\smallskip

Suppose that $\A$ is type I.  Let $\{\J_\alpha\}_{\alpha \in I}$ be a composition series for $\A$ in which $\J_{\alpha + 1}/\J_\alpha$ is a continuous trace algebra for every $\alpha \in I$ (\cite[Corollary IV.1.4.28]{B}).  Then
\begin{equation} \label{E:series}
\A^{**} \simeq \sum\nolimits^\oplus (\J_{\alpha+1}/\J_\alpha)^{**}.
\end{equation}
Because $\Dc$ is additive on direct sums and $\G$ is only subadditive (Theorem \ref{T:dsum}), to prove \eqref{E:HN'} it suffices to prove $\G \leq \Dc$ for all summands in \eqref{E:series}.  Thus we only need to show \eqref{E:HN'} for continuous trace algebras.  Since \eqref{E:HN'} is automatic for $\A^{**}$ countably-generated ($2 \leq \G(\A^{**}) \leq \aleph_0$ implies $\A^{**}$ infinite-dimensional and $\Dc(\A^{**}) \geq \aleph_0$), we may and do assume that $\G(\A^{**}) > \aleph_0$.

Continuous trace algebras are type I, so $\A^{**}$ is a type I von Neumann algebra.  Write $\A^{**} \simeq \sum_{\kappa \in K}^\oplus \B(\ell^2_\kappa) \bar{\otimes} \z_\kappa$, where the $\z_\kappa$ are abelian von Neumann algebras.  Let $\kappa_0 = \sup K$.  Now $\Dc(\A^{**})$ dominates both $\kappa_0$ and $\Dc(\z(\A^{**}))$; since at least one of these must be infinite, we get
\begin{equation} \label{E:first}
\Dc(\A^{**}) \geq \kappa_0 \cdot \Dc(\z(\A^{**})).
\end{equation}
Working from the other direction, compute
\begin{equation} \label{E:second}
\G(\A^{**}) = \max\{\log_\C(|K|), \sup \G(\B(\ell^2_\kappa) \bar{\otimes} \z_\kappa)\} \leq \kappa_0 \cdot (\kappa_0 \cdot \G(\z(\A^{**}))) = \kappa_0 \cdot \G(\z(\A^{**})),
\end{equation}
using Theorem \ref{T:dsum} for the first relation and infiniteness of $\kappa_0 \cdot \kappa_0 \cdot \G(\z(\A^{**}))$ for the last.  (According to \eqref{E:second}, it dominates $\G(\A^{**})$.)  Putting \eqref{E:first} and \eqref{E:second} together, \eqref{E:HN'} will follow if we show $\G(\z(\A^{**})) \leq \Dc(\z(\A^{**}))$.

Since $\A$ is continuous trace, $\z(\A^{**}) \simeq C_0(\hat{\A})^{**}$ (\cite[Theorem 6.3]{Pe2}).  This means that we only need to show \eqref{E:HN'} for abelian $\A$.  We write $\A = C_0(X)$ for some infinite locally compact Hausdorff space $X$.  For each unequal pair $x, y \in X$, there is a function $f_{x,y} \in C_0(X)$ with $f_{x,y}(x) \neq f_{x,y}(y)$.  By the Stone-Weierstrass theorem the family of $|X|^2 + 1 = |X|$ functions $\{f_{x,y}\} \cup\{1\}$ generates $C_0(X)$ as a $C^*$-algebra and thus $C_0(X)^{**}$ as a von Neumann algebra, making $\G(C_0(X)^{**}) \leq |X|$.  On the other hand, the points of $X$ give disjoint one-dimensional representations of $C_0(X)$, which correspond to one-dimensional summands in $C_0(X)^{**}$.  This entails $\Dc(C_0(X)^{**}) \geq |X|$ and completes the proof of the type I case.  (The subcase just established, that double duals of abelian $C^*$-algebras are maximally decomposable, improves \cite[Corollary 5.2]{HN} and its incorporation into \cite[Theorem 5.5]{HN}.)

\smallskip

Now suppose that $\A$ is generated as a $C^*$-algebra by $\leq \C$ elements.  This implies that $\G(\A^{**}) \leq \C$, and from (1) we may also assume that $\A$ is not type I.  By Sakai's nonseparable version of Glimm's theorem (conveniently formulated in \cite[Corollary 6.7.4]{Pe}), there is a $C^*$-subalgebra $\B \subseteq \A$ with ideal $\J \lhd \B$ such that $\B/\J$ is *-isomorphic to the CAR algebra $\mathbb{M}_{2^\infty}$.  Powers showed that $\mathbb{M}_{2^\infty}$ has a continuum of nonisomorphic factor representations (\cite[Section 4]{Pow}), each of which is then a summand of $\mathbb{M}_{2^\infty}^{**}$.  From all this we deduce
$$\A^{**} \supseteq \B^{**} \simeq \J^{**} \oplus (\B/\J)^{**} \supseteq (\B/\J)^{**} \simeq \mathbb{M}_{2^\infty}^{**} \qquad \Rightarrow \qquad \Dc(\A^{**}) \geq \Dc(\mathbb{M}_{2^\infty}^{**}) \geq \C. \qedhere.$$

%Results of Elliott (\cite[Theorem 2.1 and Lemma 3.3]{E}) then show that the double dual of $\mathbb{M}_{2^\infty}$ (or any separable non-type I $C^*$-algebra) has decomposability number at least $\C$, so $\Dc(\A^{**}) \geq \Dc(\mathbb{M}_{2^\infty}^{**}) \geq \C$.  (Throughout Elliott's paper the continuum hypothesis is assumed, but its use at the end of the proof of Lemma 3.3 can be avoided.  There it is noted that a set of full measure in a standard nonatomic Borel measure space cannot be countable, and so by the continuum hypothesis it has cardinality $\C$.  But if we take the space to be $([0,1], m)$, as we may, then by a fact of Lebesgue theory, sets of positive measure automatically have cardinality $\C$ (\cite[Statement $(iv)$ on p.143]{GO}).)
\end{proof}

Proposition \ref{T:last} is easily extended to a stronger fact.  Let $I(\A)$ denote the largest type I ideal of a $C^*$-algebra $\A$ (\cite[Section IV.1.1.12]{B}).  Then $\A$ satisfies \eqref{E:HN} whenever $\G((\A / I(\A))^{**}) \leq \C$, by applying Theorem \ref{T:dsum} and Proposition \ref{T:last} to the decomposition $\A^{**} \simeq I(\A)^{**} \oplus (\A / I(\A))^{**}$.

%It will also be useful to note that $\G(\A^{**})$, the number of elements needed to generate $\A^{**}$ as a von Neumann algebra, is not larger than the number of elements needed to generate $\A$ as a $C^*$-algebra.  So a \textit{sufficient} condition for \eqref{E:HN'} (or \eqref{E:HN}) is
%\begin{equation} \label{E:q}
%\text{``"}
%\end{equation}

%Unlike \eqref{E:HN}, conditions \eqref{E:HN'} and \eqref{E:q} hold for finite-dimensional algebras.  Moreover, \eqref{E:q} is easily seen to hold for every infinite-dimensional \textit{separable} $\A$.  For separability of $\A$, which is much less restrictive than separability of $\A^*$ (\cite{T}), implies countable generation, and as before $\Dc(\A^{**}) \geq \aleph_0$.

%A proof that \eqref{E:HN}/\eqref{E:HN'} is universal may not be difficult, but it would say something about the structure of double duals of arbitrarily large $C^*$-algebras, and that is territory little explored.

%The following question, which asks for more than universality of \eqref{E:HN}, is also interesting: Is every $C^*$-algebra $\A$ generated (as a $C^*$-algebra) by $\leq \Dc(\A^{**})$ elements?

\section{Cardinality of a von Neumann algebra}

One of the goals of this paper is to demonstrate that many cardinal invariants for von Neumann algebras can be expressed in terms of $\G$ and $\Dc$, mostly based on Theorem \ref{T:estimate}.  This is true for the cardinality of the predual, but in this section we show that it is \textit{not} true for the cardinality of the algebra itself.  Nonetheless the situation is not so bad: there is a simple formula that works unless the algebra is fantastically large.  (See the third condition in Theorem \ref{T:card}(2).  We mean large to an analyst, maybe not to a set theorist.)

We will go to the trouble of determining an exact cardinal bound for this phenomenon, so let us first review some notation and facts regarding cardinal arithmetic.  For a cardinal $\kappa$ and ordinal $\xi$, $\kappa^{+\xi}$ denotes the ``$\xi$th successor of $\kappa$," i.e., if $\kappa = \aleph_\alpha$, then $\kappa^{+\xi} = \aleph_{\alpha+\xi}$.  For infinite cardinals $\kappa$ and $\lambda$, the value of $\kappa^\lambda$ is determined by the following iterative scheme (\cite[Theorem 5.20]{J}):
\begin{itemize}
\item if there is $\mu < \kappa$ with $\mu^\lambda \geq \kappa$, then $\kappa^\lambda = \mu^\lambda$ (so in particular $\kappa^\lambda = 2^\lambda$ when $\kappa \leq 2^\lambda$);
\item otherwise $\text{cf}\;\kappa > \lambda \; \Rightarrow \; \kappa^\lambda = \kappa$ and $\text{cf}\;\kappa \leq \lambda \; \Rightarrow \; \kappa^\lambda = \kappa^{\text{cf}\;\kappa}$.
\end{itemize}
Here $\text{cf}\;\kappa$ is the \textit{cofinality} of $\kappa$, the least cardinality of a set of cardinals $< \kappa$ that sum to $\kappa$.  From K\"{o}nig's theorem we always have $\kappa^{\text{cf}\;\kappa} > \kappa$ (\cite[Corollary 5.14]{J}).

We thank Ilijas Farah for initially pointing out that part (1) of the next lemma is true.

\begin{lemma} \label{T:sing} ${}$
\begin{enumerate}
\item There exist cardinals $\kappa, \lambda$ such that $\kappa^\lambda > 2^\lambda \cdot \kappa^{\aleph_0}$.
\item If $\kappa^\lambda > 2^\lambda \cdot \kappa^{\aleph_0}$, then $\kappa \geq (2^{\aleph_1})^{+ \omega_1}$.
\end{enumerate}
\end{lemma}

\begin{proof}
Let $\{\kappa_\xi\}_{\xi < \omega_1}$ be a sequence of cardinals greater than $2^{\aleph_1}$ such that $\kappa_\eta > \kappa_\xi^{\omega_0}$ whenever $\eta > \xi$.  Now set $\kappa = \sup\{\kappa_\xi\}$ and $\lambda = \aleph_1$; from $\text{cf}\;\kappa = \aleph_1$ we compute $\kappa^{\aleph_1} > \kappa = \kappa^{\aleph_0} > 2^{\aleph_1}$.  This establishes (1).

We now turn to (2). From $\kappa^\lambda > \kappa^{\aleph_0}$ we have $\lambda \geq \aleph_1$, while from $\kappa^\lambda > 2^\lambda = (2^\lambda)^\lambda$ we conclude $\kappa > 2^\lambda$.  Taking $\lambda = \aleph_1$ for now, we are looking for the least $\kappa > 2^{\aleph_1}$ such that $\kappa^{\aleph_1} > \kappa^{\aleph_0}$.

Obviously $(2^{\aleph_1})^{\aleph_1} = (2^{\aleph_1})^{\aleph_0} = 2^{\aleph_1}$, and moreover $((2^{\aleph_1})^{+ k})^{\aleph_1} = ((2^{\aleph_1})^{+ k})^{\aleph_0} = (2^{\aleph_1})^{+ k}$ for every finite $k$.  Now $((2^{\aleph_1})^{+ \omega})^{\aleph_1}$ and $((2^{\aleph_1})^{+ \omega})^{\aleph_0}$ are larger than $(2^{\aleph_1})^{+ \omega}$, whose cofinality is $\aleph_0$, but they are still equal.  We may continue to argue by induction that $((2^{\aleph_1})^{+ \xi})^{\aleph_1}= ((2^{\aleph_1})^{+ \xi})^{\aleph_0}$ for any countable ordinal $\xi$.  For if $\xi$ were the lowest counterexample, the iterative scheme implies that $\text{cf}\;(2^{\aleph_1})^{+ \xi} = \aleph_1$, but this is impossible.  (If $\xi$ is a successor ordinal, $(2^{\aleph_1})^{+ \xi}$ is its own cofinality, and otherwise the cofinality is $\aleph_0$.)

This identifies $(2^{\aleph_1})^{+ \omega_1}$ as the smallest possibility for $\kappa$ when $\lambda = \aleph_1$.  The same argument for a larger $\lambda$ shows that $\kappa \geq (2^\lambda)^{+ \xi'}$, where $\xi'$ is the least ordinal of cardinality $\lambda$.  Writing $2^{\aleph_1} = \aleph_\alpha$ and $2^\lambda=\aleph_{\alpha'}$, we have $\alpha' \geq \alpha$ and $\xi' > \omega_1$ as ordinals.  This entails that $\alpha' + \xi' > \alpha + \omega_1$.  (Otherwise $\alpha'+\xi'$ would be isomorphic to an initial segment of $\alpha + \omega_1$; this would carry $\alpha'$ to an initial segment containing $\alpha$, but $\xi'$ cannot embed into the remainder even as a set, having cardinality larger than $\aleph_1$.)  We conclude that $(2^\lambda)^{+\xi'}$ is greater than $(2^{\aleph_1})^{+ \omega_1}$, so the latter is a universal strict lower bound for $\kappa$ when $\lambda > \aleph_1$.
\end{proof}

\begin{remark}
It is consistent with ZFC that $\kappa^\lambda > 2^\lambda \cdot \kappa^{\aleph_0}$ for $\kappa = (2^{\aleph_1})^{+ \omega_1}$ and $\lambda = \aleph_1$.  This follows, for instance, from the Generalized Continuum Hypothesis, or even just the Singular Cardinal Hypothesis (\cite[Theorem 5.22]{J}). %The author does not know (but doubts) whether this inequality is provable in ZFC.
\end{remark}

\begin{theorem} \label{T:card}
Let $\M$ be a von Neumann algebra.
\begin{enumerate}
\item We have
\begin{equation} \label{E:card}
|\M| \leq (\aleph_0 \cdot \G(\M))^{\aleph_0 \cdot \textnormal{dec}(\M)}.
\end{equation}
\item The inequality \eqref{E:card} is an equality whenever any of the following conditions hold:
\begin{itemize}
\item $\M$ is $\sigma$-finite;
\item $\M$ is a factor; or
\item $\G'(\M) < (2^{\aleph_1})^{+ \omega_1}$, i.e., $\M$ can be written as a direct sum of algebras each of which can be generated by fewer than $(2^{\aleph_1})^{+ \omega_1}$ elements.
\end{itemize}
\item In general the cardinality of $\M$ is not determined by $\G(\M)$ and $\Dc(\M)$.
\end{enumerate}
\end{theorem}

\begin{proof}
We start with the elementary observation that $|\M| = |\M_{\leq 1}|$.  One justification is as follows:
$$|\M_{\leq 1}| \leq |\M| = |\cup_{n \in \mathbb{N}} \M_{\leq n}| \leq \sum |\M_{\leq n}| = \aleph_0 \cdot |\M_{\leq 1}| = |\M_{\leq 1}|.$$
We use this freely in the rest of the proof.

First suppose that $\M$ is $\sigma$-finite.  Then there is a faithful normal state $\varphi$ on $\M$, and the strong topology on $\M_{\leq 1}$ is induced by the norm $\|x\|_\varphi = \varphi(x^*x)^{1/2}$ (\cite[Proposition III.2.2.7]{B}).  From Theorem \ref{T:estimate}(1) we know $s\text{-}\D(\M_{\leq 1}) = \aleph_0 \cdot \G(\M)$.  Arguing just as in \cite[Lemma 2]{Kr}, it follows that $|\M_{\leq 1}| = (\aleph_0 \cdot \G(\M))^{\aleph_0}$.   The $\sigma$-finiteness of $\M$ makes \eqref{E:card} an equality.

Next assume that $\M$ can be written as $\B(\ell^2_\mu) \bar{\otimes} \N$, where $\N$ is $\sigma$-finite and $\mu$ is either 1 or uncountable.  In particular any factor can be put in this form.  By Example \ref{T:typeI} and Corollary \ref{T:mult}(2),  $\G(\M) = \mu \cdot \G(\N)$; by Proposition \ref{T:mult2}, $\Dc(\M) = \mu \cdot \Dc(\N)$.  We also have that $|\M| = |\N|^\mu$: the relation $\leq$ follows from the fact that every element of $\M$ can be represented as a matrix of $\mu^2 = \mu$ entries in $\N$, and the relation $\geq$ follows from the fact that $\M_{\leq 1}$ contains the unit ball of the diagonal algebra $\ell^\infty_\mu(\N)$, which has cardinality $|\N|^\mu$.  Using the previous paragraph, we again obtain equality in \eqref{E:card}:
$$|\M| = |\N|^\mu = (\aleph_0 \cdot \G(\N))^{\aleph_0 \cdot \textnormal{dec}(\N) \cdot \mu} = (\aleph_0 \cdot \mu \cdot \G(\N))^{\aleph_0 \cdot \textnormal{dec}(\N) \cdot \mu} = (\aleph_0 \cdot \G(\M))^{\aleph_0 \cdot \textnormal{dec}(\M)}.$$
Here is the justification for changing the base expression in the third equality.  If $\mu \leq \aleph_0 \cdot \G(\N)$, the base has not changed.  Otherwise $\mu$ must be uncountable, so by $\sigma$-finiteness of $\N$ the exponent in these expressions is just $\mu$, while the bases are infinite cardinals $\leq \mu$.

Now it is a fact of dimension theory that any von Neumann algebra $\M$ can be written as a direct sum of algebras $\M_i = \B(\ell^2_{\mu_i}) \bar{\otimes} \N_i$, where the $\mu_i$ and $\N_i$ are as in the previous paragraph (see, e.g., \cite[Theorem 2.5]{S}).  For the left-hand side of \eqref{E:card} we get
\begin{align} \label{E:last1}
|\M| &= |\M_{\leq 1}| = \prod |(\M_i)_{\leq 1}| = \prod (\aleph_0 \cdot \G(\M_i))^{\aleph_0 \cdot \textnormal{dec}(\M_i)} \\ \notag &\leq \prod (\aleph_0 \cdot \sup \G(\M_i))^{\aleph_0 \cdot \textnormal{dec}(\M_i)} = (\aleph_0 \cdot \sup \G(\M_i))^{\aleph_0 \cdot \sum \textnormal{dec}(\M_i)}.
\end{align}
We use Theorem \ref{T:dsum} to compute the right-hand side of \eqref{E:card} as follows:
\begin{align} \label{E:last2}
(\aleph_0 \cdot \G(\M))^{\aleph_0 \cdot \textnormal{dec}(\M)} &= (\aleph_0 \cdot \max\{\sup \G(\M_i), \log_\C(|I|)\})^{\aleph_0 \cdot \sum \textnormal{dec}(\M_i)} \\ \notag &= (\aleph_0 \cdot \sup \G(\M_i))^{\aleph_0 \cdot \sum \textnormal{dec}(\M_i)}.
\end{align}
The second equality is justified similarly to the end of the previous paragraph.  If the base really changed, then $\log_\C(|I|)$ would have to be uncountable; but then both bases are infinite and dominated by the exponent (which is at least $|I|$), so the quantities are equal.  Since \eqref{E:last1} and \eqref{E:last2} end with equal expressions, we obtain part (1) of the theorem and also deduce that \eqref{E:card} is an equality whenever \eqref{E:last1} is an equality.

In fact \eqref{E:last1} can be a strict inequality, but then some big cardinals must be involved.  Set $\kappa = \G'(\M)$.  Write $\M$ as a direct sum of algebras each generated by $\leq \kappa$ elements, then decompose each summand as in the previous paragraph; thus $\M = \sum^\oplus \M_i$ as in the previous paragraph, with the additional condition that $\G(\M_i) \leq \kappa$ for all $i$.  Set $\lambda = \aleph_0 \cdot \sum \Dc(\M_i) = \aleph_0 \cdot \Dc(\M)$.  Assuming the inequality between the two terms in \eqref{E:last1} is strict, we estimate
\begin{align*}
\kappa^\lambda &= \prod (\aleph_0 \cdot \sup \G(\M_i))^{\aleph_0 \cdot \textnormal{dec}(\M_i)} > \prod (\aleph_0 \cdot \G(\M_i))^{\aleph_0 \cdot \textnormal{dec}(\M_i)} \\ &= \left[\prod (\aleph_0 \cdot \G(\M_i))^{\aleph_0 \cdot \textnormal{dec}(\M_i)}\right]^{\aleph_0} \cdot \left[\prod (\aleph_0 \cdot \G(\M_i))^{\aleph_0 \cdot \textnormal{dec}(\M_i)}\right] \geq \kappa^{\aleph_0} \cdot \aleph_0^\lambda = \kappa^{\aleph_0} \cdot 2^\lambda.
\end{align*}
From Lemma \ref{T:sing}(2) this implies $\kappa \geq (2^{\aleph_1})^{+ \omega_1}$, finishing part (2) of the theorem.

Finally, use Lemma \ref{T:sing}(1) to find $\kappa$ and $\lambda$ such that $\kappa^\lambda > 2^\lambda \cdot \kappa^{\aleph_0}$.  Let
$$\M_1 = \B(\ell^2_\lambda) \bar{\otimes} L(\F_\kappa), \qquad \M_2 = L(\F_\kappa) \oplus \ell^\infty_\lambda, \qquad \M_3 = L(\F_\kappa) \oplus (\B(\ell^2_\lambda) \bar{\otimes} L(\F_\lambda)).$$
From Example \ref{T:typeI}, Theorem \ref{T:dsum}, Example \ref{T:LG}(1), and Corollary \ref{T:mult}(2) we have $\G(\M_j) = \kappa$ and $\Dc(\M_j) = \lambda$.  But the first part of \eqref{E:last1} gives
$$|\M_1| = \kappa^\lambda > \kappa^{\aleph_0} \cdot \aleph_0^\lambda = |\M_2| = |\M_3|.$$
This establishes part (3) of the theorem.  We exhibited both $\M_2$ and $\M_3$ because they endow the second and third conditions in part (2) with some sharpness: equality in \eqref{E:card} follows from $\Dc(\M) \leq \aleph_0$ or $\Dc(\z(\M)) \leq 1$, but it does not follow from $\Dc(\M) \leq \aleph_1$, $\Dc'(\M) \leq \aleph_0$, or $\Dc(\z(\M)) \leq 2$.
\end{proof}

\smallskip

\textbf{Acknowledgments.}  The author thanks Ilijas Farah, Takeshi Katsura, Nik Weaver, and Stuart White for valuable input, and Kai-Uwe Bux for linguistic assistance with the article \cite{vN}.

%CHECKLIST: which/that, layout, stratila-zsido 7.2 "piecewise of countable type", can I get $S$ within a single genus, can't change numbering of 3.4 and 3.5, remember "newpage", send to Nik, Chuck, Ilijas (ask about generators for ultrapowers), Matthias, Zhiguo, Katsura, Saito, Tom, Craig; Takeshi suggests local versions of heartsuit and Theorem 4.3

\end{document}